\documentclass[10pt,leqno]{amsart}
\usepackage{amsmath,amsfonts,amssymb,amsthm} 
\usepackage{epsfig}
\usepackage{graphicx}
\usepackage{psfrag} 
\newtheorem{theorem}{Theorem}[section]
\newtheorem{thm}{Theorem}[section]
\newtheorem{lemma}[theorem]{Lemma}

\newtheorem{proposition}[theorem]{Proposition}

\theoremstyle{definition}
\newtheorem{definition}[theorem]{Definition}

\newtheorem{remark}[theorem]{Remark}

\numberwithin{equation}{section}

\newcommand{\del}{\nabla}

\newcommand{\cC}{{\mathcal C}}

\newcommand{\cH}{{\mathcal H}}

\newcommand{\cQ}{{\mathcal Q}}
\newcommand{\cR}{{\mathcal R}}

\newcommand{\field}[1]{\mathbb{#1}}
\newcommand{\N}{\field{N}}          		
\newcommand{\Z}{\field{Z}}          		
\newcommand{\R}{\field{R}}          		

\newcommand{\mass}{{\mathbf{M}}}

\newcommand{\abs}[1]{\lvert #1 \rvert}
\newcommand{\norm}[1]{\lVert #1 \rVert}

\newcommand{\loc}{{\rm loc}}

\newcommand\ang{\sphericalangle}

\DeclareMathOperator{\supp}{supp}

\DeclareMathOperator{\dist}{dist}

\DeclareMathOperator\diver{div}

\DeclareMathOperator{\Vol}{Vol}

\DeclareMathOperator\spt{spt}

\DeclareMathOperator\Lip{Lip}

\DeclareMathOperator{\bd}{bd}
\DeclareMathOperator{\cl}{cl}
\DeclareMathOperator{\ir}{int}
\DeclareMathOperator\Int{Int}
\DeclareMathOperator\cone{Cone}
\let\pois\setminus

\let\cal\mathcal

\begin{document}
\title[Convexity at infinity]{Convexity at infinity in Cartan-Hadamard manifolds and applications to the asymptotic Dirichlet and Plateau problems}
\author[Jean-Baptiste Casteras]{Jean-Baptiste Casteras}
\address{D\'epartement de Math\'ematique, Universit\'e Libre de Bruxelles, CP 214, Boulevard du Triomphe, B-1050 Bruxelles, Belgium}
\email{jeanbaptiste.casteras@gmail.com}
\author[Ilkka Holopainen]{Ilkka Holopainen}
\address{Department of Mathematics and Statistics, P.O. Box 68, 00014 University of
Helsinki, Finland}
\email{ilkka.holopainen@helsinki.fi}
\author[Jaime B. Ripoll]{Jaime B. Ripoll}
\address{UFRGS, Instituto de Matem\'atica, Av. Bento Goncalves 9500, 91540-000 Porto
Alegre-RS, Brasil}
\email{jaime.ripoll@ufrgs.br}
\thanks{J.-B.C. supported by the FNRS project MIS F.4508.14; I.H. supported by the Academy of 
Finland, project 252293; J.R. supported by the CNPq (Brazil) project 302955/2011-9}
\subjclass[2000]{53C20, 53C21, 58J32}
\keywords{Hadamard manifolds, asymptotic Dirichlet problem, asymptotic Plateau problem}

\begin{abstract}
We study the asymptotic Dirichlet and Plateau problems on Cartan-Hadamard manifolds satisfying the so-called Strict Convexity (abbr. SC) condition.
The main part of the paper consists in studying the SC condition on a manifold whose sectional curvatures are bounded from above and below by certain 
functions depending on the distance to a fixed point. In particular, we are able to verify the SC condition on manifolds whose curvature 
lower bound can go to $-\infty$ and upper bound to $0$ simultaneously at certain rates, or on some manifolds whose sectional curvatures go to 
$-\infty$ faster than any prescribed rate. These improve previous results of Anderson, Borb\'ely, and Ripoll and Telichevsky.
We then solve the asymptotic Plateau problem for locally rectifiable currents with $\Z_2$-multiplicity in a Cartan-Hadamard manifold satisfying the SC 
condition given any compact topologically embedded $(k-1)$-dimensional submanifold of $\partial_{\infty}M,\ 2\le k\le n-1$, as the boundary data. We also 
solve the asymptotic Plateau problem for locally rectifiable currents with $\Z$-multiplicity on any rotationally symmetric manifold satisfying the SC 
condition given a smoothly embedded submanifold as the boundary data. These generalize previous results of Anderson, Bangert, and Lang.  Moreover, we 
obtain new results on the asymptotic Dirichlet problem for a large class of PDEs. In particular, we are able to prove the solvability of this problem on 
manifolds with super-exponential decay (to $-\infty$) of the curvature.
\end{abstract}

\maketitle

\section{Introduction}In this paper we investigate the notion of convexity at infinity and its applications for a class of Cartan-Hadamard manifolds. We recall that a Cartan-Hadamard manifold $M$ is a complete, connected, and simply connected Riemannian $n$-manifold, $n \geq 2$, of non-positive sectional curvature. It is well-known that a Cartan-Hadamard manifold $M$ can be compactified in the cone topology by adding a sphere at infinity, also called the asymptotic boundary of $M$; we refer to \cite{EO} for more details. In the following, we will denote by $\partial_\infty M$ the sphere at infinity and 
by $\overline{M}=M\cup\partial_\infty M$ the compactification of $M$.

It is well-known that the notion of convexity at infinity plays an important role in the solvability of the asymptotic Dirichlet and Plateau problems. Let 
us briefly introduce these two problems. The asymptotic Plateau problem in $M$ basically consists in finding an 
absolutely area minimizing $k$-dimensional submanifold $\Sigma\subset M$ asymptotic to given $(k-1)$-dimensional submanifold  
$\Gamma\subset\partial_{\infty}M,\ \ 2\leq k\leq \dim M-1$.
If such $\Sigma$ exists we say that the asymptotic Plateau problem is solvable in $M$ for $\Gamma$. 
On the other hand, the asymptotic Dirichlet problem on $M$ for an operator $\cQ$ is the following: 
Given a continuous function
$h$ on $\partial_\infty M$, does there exist a (unique) function $u\in C(\overline{M})$ such that $\cQ [u] = 0$ on
$M$ and $u|\partial_\infty M = h$? 

The connection between convexity at infinity and the solvability of the Dirichlet problem was first pointed out by Choi in \cite{choi} 
(see Remark~\ref{kas}). 
He proved that if $M$ satisfies the "convex conic neighborhood condition" (see the definition below) and if the sectional curvature $K_M$ of $M$ 
is bounded from above by a negative constant, then the asymptotic Dirichlet problem for the Laplacian admits a solution. Let us recall Choi's convexity condition:

\begin{definition}[The convex conic neighborhood condition]
Let $M$ be a Hadamard manifold. We say that $M$ satisfies the convex conic neighborhood condition if, given $x\in \partial_\infty M$, for any $y\in \partial_\infty M\setminus\{x\}$, there
exist $V_x\subset \overline{M}$, a neighborhood of $x$, and $V_y\subset \overline{M}$, a neighborhood of $y$ such that $V_x$ and $V_y$ are disjoint open sets of $\overline{M}$ in terms of the cone topology and $V_x \cap M$ is convex with $C^2$ boundary.
\end{definition}

The use of the convex conic neighborhood condition by Choi for solving the asymptotic Dirichlet problem for the Laplacian is heavily based on the linearity 
of the Laplace operator and does not apply directly to other PDEs like the $p$-Laplacian and the minimal graph equation. To deal with more general PDEs, 
Ripoll and Telichevesky (\cite{RT}) recently introduced another notion of convexity at infinity which was called the Strict Convexity condition. 
Intuitively, this condition insures that  we can extract from $\overline{M}$ a neighborhood of any point of the sphere at infinity so that the remaining is 
convex. The precise definition of this notion is the following:
\begin{definition}[SC condition]
Let $M$ be a Hadamard manifold. We say that $M$ satisfies the \emph{Strict
Convexity condition (SC condition)} if, given $x\in\partial_{\infty}M$ and a
relatively open subset $W\subset\partial_{\infty}M$ containing $x,$ there
exists a $C^{2}$ open subset $\Omega\subset M$ such that
$x\in\ir \partial_{\infty}\Omega  \subset W$ and $M\setminus\Omega$ is convex. Here $\ir \partial_{\infty}\Omega$ denotes
the interior of $\partial_{\infty}\Omega$ in $\partial_{\infty}M$.
\end{definition}

It turns out that the SC condition can be used to prove the solvability of the asymptotic Dirichlet problem for a large class of PDEs which includes, besides the Laplacian, the $p$-Laplacian and the minimal graph equation. 
Now we would like to give a brief overview of the earlier results concerning the class of manifolds where either the SC or the convex conic neighborhood conditions hold. 

The first result in this direction was obtained by Anderson in \cite{andJDG}. There he proved that the convex conic neighborhood condition is satisfied on a manifold $M$ if the sectional curvatures are pinched i.e. $-k_1\leq K_M\leq -k_2$, for some constants $k_1\geq k_2>0$. A few years later,  Borb\'ely in \cite{borbpams} improved the curvature assumption of the previous result by modifying the construction of Anderson. Namely he assumed that $K_M\leq -1$ and
\[
\min_{\rho (x) \leq R} |K_M (x)|\leq e^{\lambda \rho (x)},\ \text{if}\ R>R_0\ \text{and for some}\ \lambda<\dfrac{1}{3}, 
\]
where $\rho (x)$ stands for the distance of $x\in M$ to a fixed point of $M$. Concerning the SC condition, Ripoll and Telichevesky in \cite{RT} proved that it is satisfied on rotationally symmetric manifolds satisfying $K_M\leq -k^2$, $k>0$. They also adapted the construction of Anderson and proved that the SC condition holds if
\[
 -\dfrac{e^{2k\rho(x)}}{\rho(x)^{2+2\varepsilon}}\leq K_M (x) \leq -k^2,
 \]
for all $x\in M$ such that $\rho (x)=d(x,o)\geq R^\ast $, $R^\ast$ large enough, and where $k$ and $\varepsilon$ are positive constants. \\

The first part of this paper is devoted to generalize these previous results in two directions: we allow the lower bound of the sectional curvature to go 
to $0$ at a certain rate and, on the other hand, we allow it to go to $-\infty$. Before giving our precise assumptions on the sectional curvature, we state two results as special cases of our main theorem.
\begin{thm}
\label{thmSCex1}
Let $M$ be a Hadamard manifold and let $o\in M$. Suppose that there exists $R^\ast >0$ such that, for some constants 
$\varepsilon>\tilde{\varepsilon}>0$,
 \begin{equation*}
-\dfrac{(\log \rho(x))^{2\tilde{\varepsilon}}}{\rho(x)^2} \le K_M(x) \le -\dfrac{1+\varepsilon}{ \rho(x)^2\log \rho(x)},
\end{equation*}
for all $x\in M$, with $\rho (x)=d(x,o)\geq R^\ast $. Then $M$ satisfies the\textrm{ SC} condition.
\end{thm}

\begin{thm}
\label{thmSCex2}
Let $M$ be a Hadamard manifold and let $o\in M$. Suppose that there exists $R^\ast >0$ such that, for some constants 
$\phi>1/4,\ \epsilon>0$, and $c>0$,
 \begin{equation*}
-c\,e^{(2-\epsilon)\rho(x)}e^{\displaystyle{e^{\rho(x)/e^3}}} \le K_M(x) \le -\phi e^{2\rho (x)},
\end{equation*}
for all $x\in M$, with $\rho (x)=d(x,o)\geq R^\ast $. Then $M$ satisfies the\textrm{ SC} condition.
\end{thm}

More generally, our main result is the following (see \eqref{Jacobi_eq} for the definition of $f_a$):
 
\begin{thm}
\label{thmSC}
Let $M$ be a Hadamard manifold and let $o\in M$. Suppose that there exists $R^\ast >0$ such that
 \begin{equation}\label{aRcondition1}
-b^{2} (\rho (x)) \le K_M(x) \le -a^2(\rho (x)),
\end{equation}
for all $x\in M$, with $\rho (x)=d(x,o)\geq R^\ast $ and $b$ monotonic. Moreover, assume that at least one of the following alternatives holds:
\begin{enumerate}
\item for some constants  $C_1>0,\ \tilde{\varepsilon},2\alpha\in (0,\varepsilon),\ t_0>0$, and $0<\lambda<1$, we have for $t$ large enough
that
\[
 a^2(t)\geq \dfrac{1+\varepsilon}{ t^2\log t},
 \] 
$b(t/2)\le C_1 b(t)$, $b(t+1)\leq C_1 b(t)$, and 
\begin{equation}
\label{pinch1}
\frac{(\log t)^{\tilde{\varepsilon}}}{t}\le b(t) \le 
\begin{cases}
\dfrac{f_a^\prime (t)}{t (\log t)^{1+2\alpha}}\dfrac{f_a (\lambda t)}{f_a (t)},&\\&\\
\dfrac{f_a^\prime (t)}{t (\log t)^{1+2\alpha}}\dfrac{f_a (t-t_0)}{f_a (t)},\text{ if, in addition, $b$ is increasing};
\end{cases}
\end{equation}
\item or $a$ and $b$ are increasing continuous functions such that $b(0)>0$ and
\begin{equation}
\label{pinch}
\lim_{t\to\infty}
\dfrac{t (\log t)^{1+\varepsilon}f_a( t-2) b(t)}{f_a^\prime (t-2) f_a (t-3)}<\infty .
\end{equation}
\end{enumerate}
Then $M$ satisfies the\textrm{ SC} condition.
\end{thm}
We observe that Theorem \ref{thmSCex1} is obtained from Theorem \ref{thmSC} by using Proposition \ref{jacest}. Also Theorem \ref{thmSCex2} follows from Theorem \ref{thmSC} by taking $f_a(t)=\sinh^{\circ 2}(t)=\sinh (\sinh t)$; a direct computation shows that $-a^2(t)=-\sinh t \coth (\sinh t)-\cosh^2 (t)$.
It is worth noting that Theorem~\ref{thmSCex2} is just an example to illustrate the use of Theorem~\ref{thmSC}. Choosing $f_a=\sinh^{\circ m}$ to be the 
$m$-th iterate of $\sinh$, we can make the ratio $b^2(t)/a^2(t)$ of the curvature bounds grow faster than $\exp^{\circ k}$ for any given $k\in\N$ by choosing $m$ large enough. 
We point out that the pinching conditions \eqref{pinch1} and \eqref{pinch} are (almost) the same in their shared range. Let us also observe that the curvature lower bound and the pinching condition appearing in  Theorem \ref{thmSCex1} are the same as the ones obtained by the authors in \cite{CHR1} to solve the asymptotic Dirichlet problem for a large class of operators. 
The upper bound for the sectional curvature in Theorem \ref{thmSCex1}  is interesting since it is close to optimal for the solvability of the asymptotic Dirichlet problem: Indeed, if we assume 
that 
\[
K(x)\geq-\frac{1}{\rho(x)^{2}\log \rho(x)}
\] 
then all bounded harmonic functions are constant in dimension 2 and, more generally in any dimensions $n$, all bounded $p-$harmonic functions, with 
$p\ge n$, are constant; see the discussion in \cite{CHR1}. This may indicate that the upper bound for the sectional curvature in Theorem \ref{thmSCex1} 
could be optimal for the SC condition. 

The proofs of the cases (1) and (2) in Theorem \ref{thmSC} are quite different depending on the case we consider. Under the assumption (1), our proof relies heavily on the construction of smooth extensions of certain ``angular functions". In the case (2), we adapt the approach of Anderson. Let us notice that in the second case, due to the fact that our construction relies on an iterative procedure, it seems unavoidable that the condition \ref{pinch} involves terms depending on different radii. However, it is possible to obtain a more technical condition where the radii tend to each others. This is the object of Remarks \ref{rmk1} and \ref{rmk2}. 
 
 A natural question is to investigate a relation between the SC condition and Choi's convex conic neighorhood condition. In the forthcoming paper \cite{CHRv}, we prove that the SC condition and the convex conic neighborhood condition are equivalent provided that the sectional curvature of $M$ satisfies 
 \[
 K_M(x) \le -\dfrac{1+\varepsilon}{ \rho(x)^2(\log \rho(x))}
 \] 
 for some $\varepsilon>0$ and $\rho(x)$ large enough. We conjecture that these conditions are actually equivalent.

In the same paper, we will also investigate the optimality of the pinching conditions \eqref{pinch1} and \eqref{pinch} obtained in Theorem \ref{thmSC}. It 
is well-known that some kind of pinching conditions for the curvature are necessary for the SC condition to hold.  Independently, Ancona \cite{ancrevista} 
and Borb\'ely \cite{Bor} constructed a manifold $M$ of dimension $3$ with sectional curvatures bounded from above by $-1$ satisfying the following 
property: there exists a point $P\in\partial_\infty M$ such that the convex hull of every neighborhood $U\subset \overline{M}$ of $P$ contains the whole 
manifold $M$. Holopainen in \cite{H_ns} (see also \cite{ATU}) was able to give an estimate for the lower bound of sectional curvatures in Borb\'ely's 
example by proving that it decays to $-\infty$ faster than $-\exp (\tfrac{1}{2}\exp (2\rho(x)))$, for $\rho(x)$ large enough. Let us remark that assuming $K_M\leq -k^2$, for some 
constant $k$, we are only able to prove that the SC condition holds if the lower curvature bound tends to $-\infty$ exponentially.

Next, we give some applications of our convexity results for the asymptotic Dirichlet and Plateau problems. First let us recall the result established 
by Ripoll and Telichevesky in \cite{RT} concerning the asymptotic Dirichlet problem. There they considered a large class of elliptic differential 
operators $\cQ$ of the form
\begin{equation}\label{defopQ}
\cQ[u]=\diver\left(\frac{a(|\nabla u|)}{|\nabla u|}\nabla u\right)=0,
\end{equation}
where $a\in C^1([0,\infty))$ satisfies
\begin{itemize}
\item $a(0)=0$, $a^\prime (s)>0$ for all $s>0$;
\item $a(s)\leq C (s^{p-1}+1)$, for some constant $C$ and for all $s\in [0,\infty)$; 
\item there exist $q>0$ and $\delta>0$ such that $a(s)\geq s^q$ for all $s\in [0,\delta]$.
\end{itemize}
One notices that the minimal graph operator and the $p$-Laplacian belong to this class of operators. The main result of \cite{RT} is the following:
\begin{thm}
\label{RT4}
Let $M$ be a Cartan-Hadamard manifold with sectional curvature $K_M\leq -k^2$, $k>0$, satisfying the \emph{SC} condition. Assume, moreover, that
\begin{enumerate}
\item there exists a sequence of bounded $C^\infty$ domains $\Omega_k\subset \Omega$, $k\in \N$, satisfying $\Omega_k \subset \Omega_{k+1}$ and 
\[
\bigcup_{k\in \N} \Omega_k=M,
\] 
such that the Dirichlet problem for the operator $\cQ$ (as defined above) is solvable in $\Omega_k$ for all $C^\infty$ boundary data;
\item sequences of solutions with uniformly bounded $C^0$ norm are compact in relatively  compact subsets of $M$.
\end{enumerate}
Then the asymptotic Dirichlet problem for $\cQ$ is solvable for any continuous boundary data.
\end{thm}
As an immediate consequence of Theorem \ref{thmSC} and Theorem \ref{RT4} we obtain new solvability results for the asymptotic Dirichlet problem.  In 
particular, it follows that the asymptotic Dirichlet problem for the class of operators $\cQ$ described above is solvable for any continuous  
boundary data at infinity in  Hadamard manifolds whose sectional curvatures satisfy the condition of Theorem \ref{thmSCex2}. 
We notice that the solvability of the asymptotic Dirichlet problem under  superexponential decay of the curvature was not known earlier. 
We point out that, more generally, from the combination of the 
above mentioned theorems we are able to consider non rotationally symmetric Hadamard manifolds having curvature tending to $-\infty$ at any speed.
A negative curvature upper bound alone is not sufficient for the solvability of the asymptotic Dirichlet problem as Ancona's and Borb\'ely's examples 
(\cite{ancrevista}, \cite{Bor}) show in the case of the Laplace equation. In \cite{H_ns}, Holopainen generalized Borb\'ely's result to the $p$-Laplace 
equation, and very recently, Holopainen and Ripoll \cite{HR_ns} extended these
nonsolvability results to the operator $\cQ$ (as defined in \eqref{defopQ}), in particular, to the
minimal graph equation.

Finally, we present applications of the SC condition to the solvability of the asymptotic Plateau problem. First we introduce some basic 
terminology and notation. We denote by $\cR_{k}(U)$ and $\cR_{k}^2(U)$ the spaces of rectifiable $k$-currents on an open set 
$U\subset M$ with $\Z$- and $\Z_2$-multiplicity, respectively. 
The mass, support, and the boundary of $S\in \cR_{k}(U)\cup \cR_k^2(U)$ are denoted by $\mass(S),\ \spt S$, and $\partial S$, respectively.
For a Borel set $V\subset U$ and $S\in \cR_k(U)\cup \cR_k^2(U)$, the restriction of $S$ to $V$ is denoted by 
$S\llcorner V$ and defined as $S\llcorner V(\omega)=S(\chi_V \omega)$. The local classes $\cR_{k,\loc}(U)$ and $\cR_{k,\loc}^2(U)$ 
are defined in an obvious way.
Every $k$-rectifiable oriented set $S\subset M$, with $\cH^k(S\cap K)<\infty$ for all compacta $K\subset M$, defines a locally rectifiable 
$k$-current $[S]$ as
\[
[S](\omega)=\int_{S}\omega,
\]
where $\omega$ is a differential $k$-form.
We refer to \cite{federer}, \cite{morgan}, and \cite{simon} for standard references, see also \cite{BL} for a concise 
presentation. Before stating our results, we briefly review some previous results on this topic.
Anderson proved in \cite{AndInv} (see also \cite{AndCMH}) that the asymptotic Plateau problem is solvable in the hyperbolic space 
$\mathbb{H}^{n}$ for smoothly embedded closed submanifolds $\Gamma\subset\partial_{\infty}\mathbb{H}^{n}$. Bangert and Lang \cite{BL} extended Anderson's result to a fairly large class of manifolds.  
To state  their result, which we want to refer, let $(M,g)$ be an $n$-dimensional Cartan-Hadamard manifold with sectional curvature satisfying 
$-b^2\leq K_M\leq -1$, for some  $b\geq 1,$ and let $\tilde{g}$ be a Riemannian metric on $M$ Lipschitz equivalent to $g$, i.e. there exist two constants 
$\beta\geq \alpha>0$ such that $\alpha^2 g(v,v)\leq \tilde{g}(v,v)\leq \beta^2 g(v,v)$ for all $v\in TM.$ They proved the existence of a complete 
$\tilde{g}$-minimizing locally rectifiable $k $-current modulo two in $M $ that is asymptotic to a given compact $(k-1)$-dimensional embedded 
submanifold of $ \partial_\infty M$, with $k\in \{2,\ldots , n-1\}$. It is worth noting that on such a manifold $(M,\tilde g)$ the sectional curvature can 
easily go to  $-\infty$ or to $0$; in fact, some sectional curvatures can take arbitrary large positive values. 
In \cite{Lang} Lang extended the results of \cite {BL} by replacing the curvature lower and upper bounds of $(M,g) $ by 
weaker bi-Lipschitz invariant conditions: bounded geometry and Gromov hyperbolicity, respectively. Very recently, Ripoll and Tomi \cite{RTo} investigated 
the asymptotic Plateau problem for minimal type disks using the classical Plateau theory and
proved the existence of solutions on certain Hadamard manifolds whose curvature can go
to $-\infty$ under a prescribed bound decay of the sectional curvature.  Finally, we refer to \cite{cosku} for a survey on the asymptotic Plateau problem
and to \cite{cosku2} and \cite{KloMaz} for recent studies on the asymptotic Plateau problem in $\Bbb H^2\times\R$.


Our results on the asymptotic Plateau problem are the following:
\begin{thm}
\label{TSMC} 
Let $M^{n},$ $n\geq3,$ be a Cartan-Hadamard manifold satisfying the SC condition and let $\Gamma\subset \partial_{\infty} M^n$ be a (topologically) 
embedded closed $(k-1)$-dimensional submanifold, with $2\leq k\leq n-1$. Then there exists a complete, absolutely area minimizing, locally 
rectifiable $k$-current $\Sigma$ modulo 2 in $M^n$ asymptotic to $\Gamma$ at infinity, i.e. $\partial_{\infty}\spt\Sigma=\Gamma$. 
\end{thm}

In the case $k=n-1$ we can treat more general limit sets $\Gamma\subset\partial_{\infty}M^n$ and, moreover, work with $\Z$-multiplicity currents. 
Following \cite{BL} we denote by $\bd,\ \cl$, and $\ir$ the boundary, closure, and interior with respect to the sphere topology of 
$\partial_{\infty}M^n$. Our result is a counterpart of \cite[4.4]{BL}; see also \cite[3.2]{Lang1} and \cite[5.4]{Lang}.
\begin{thm}
\label{TSMC2} Let $M^{n},$ $n\geq 2,$ be a Cartan-Hadamard manifold satisfying the SC condition.
Suppose that $\Gamma\subset\partial_{\infty}M^n$ satisfies $\Gamma=\bd A$ for some $A\subset\partial_{\infty}M^n$ with $A=\cl(\ir A)$.
Then there exists a closed set $W\subset M^n$ of locally finite perimeter in $M^n$ such that $\Sigma:=\partial[W]\in\cR_{n-1,\loc}(M^n)$ is 
minimizing in $M^n,\ \partial_{\infty}W=A$, and $\partial_{\infty}\spt\Sigma =\Gamma$.
\end{thm}

The two theorems above are in line with those in \cite{BL}, \cite{Lang1}, and \cite{Lang}. However, there are differences.
Firstly, all the manifolds in those papers have bounded geometry (see \cite[p. 33]{Lang} for the definition) whereas this is not the case 
in our setting since we can allow the sectional curvature upper bound $-a^2(\rho(x))$ decay to $-\infty$. 
Secondly, on manifolds considered in \cite{BL}, \cite{Lang1}, and \cite{Lang}, some sectional curvatures can be $0$ (or even positive) 
outside any compact set, and therefore such manifolds are not covered by our results.

The next result is a generalization of Anderson's theorem (\cite[Theorem 3]{AndInv}) to rotationally symmetric Cartan-Hadamard manifolds satisfying the 
SC condition.

\begin{thm}
\label{TSMC3}
Let $M^{n},$ $n\geq3,$ be a rotationally symmetric Cartan-Hadamard manifold around $o\in M^n$ satisfying the SC condition. 
Identify $\partial_{\infty}M^n$ with the unit $(n-1)$-sphere $S_o M^n\subset T_o M^n$. Let $\Gamma\subset S_o M^n$ be a closed, smoothly embedded, orientable $(k-1)$-dimensional submanifold, with $2\leq k\leq n-1$. 
Then there exists a complete, absolutely area minimizing, locally rectifiable $k$-current $\Sigma$ in $M^n$ asymptotic 
to $\Gamma$ at infinity, i.e. $\partial_{\infty}\spt\Sigma=\Gamma$. 
\end{thm}
We can get rid of the assumption $M$ being rotationally symmetric by assuming that sectional curvatures satisfy \eqref{aRcondition1} with some increasing 
functions $a$ and $b$ such that $b-a^2/b$ is integrable over $[0,\infty)$. Note that the integrability condition \eqref{integrcond} allows the difference between curvature upper and lower bounds go to $\infty$. We have the following ''current version" of \cite[Theorem 3]{RTo}.
\begin{thm}
\label{TSMC4}
Let $M^{n},\ n\geq3,$ be a Cartan-Hadamard manifold satisfying the SC condition. Suppose that  
\[
-(b\circ\rho)^{2}(x)\le K(P)\le-(a\circ\rho)^{2}(x)
\]
for all $2$-dimensional subspaces $P\subset T_{x}M,\ x\in M$, where $\rho(x)=d(x,o)$ and
$a,b\colon [0,\infty)\to [0,\infty)$ are increasing functions satisfying
\begin{equation}\label{integrcond}
\int_{0}^{\infty}\frac{b^2(t)-a^2(t)}{b(t)}\,dt<\infty.
\end{equation}
Identify $\partial_{\infty}M^n$ with the unit $(n-1)$-sphere $S_o M^n\subset T_o M^n$. Let $\Gamma\subset S_o M^n$ be a closed, smoothly embedded, orientable $(k-1)$-dimensional submanifold, with $2\leq k\leq n-1$. 
Then there exists a complete, absolutely area minimizing, locally rectifiable $k$-current $\Sigma$ in $M^n$ asymptotic 
to $\Gamma$ at infinity, i.e. $\partial_{\infty}\spt\Sigma=\Gamma$. 
\end{thm}

We refer to \cite{alm}, \cite{federer2}, and the recent \cite{DeLeSp} for discussions on the regularity of area minimizing locally rectifiable currents; see also \cite{morgan} and \cite{DeLe} for surveys. In particular, a $k$-dimensional area minimizing rectifiable current (or more precisely its support) in $\R^n$ is a smooth, embedded manifold on 
the interior except for a singular set of Hausdorff dimension at most $k-2$. In the codimension $1$ case, i.e. $k=n-1$, the singular set is of Hausdorff 
dimension at most $n-8$.

The plan of this paper is the following: in Section~\ref{sec:constext}, we construct smooth extensions of certain "angular" functions with gradient and 
Hessian controlled in terms of curvature bounds. These extensions will play an important part in the proof of Theorem \ref{thmSC} under assumptions 
\eqref{pinch1}. In Section~\ref{sec:sc}, we give the proof of Theorem \ref{thmSC}.
Finally, in the last section, we prove Theorems \ref{TSMC}, \ref{TSMC2}, \ref{TSMC3}, and \ref{TSMC4}.\\
\begin{remark}\label{kas}
While completing this manuscript, we found the article \cite{kasue} by Kasue, where he introduced the following notion of convexity at infinity in order to 
solve the asymptotic Dirichlet problem for the Laplacian. Let $M$ be a Hadamard manifold. Following Kasue, we say that the condition (R) holds at a point 
$x\in\partial_\infty M$ if for any neighborhood $U$ of $x$, there exists a neighborhood $V$ of $x$ such that $V\subset U$ and $M\setminus V$ is (totally) convex.
He proved that the asymptotic Dirichlet problem for the Laplacian is solvable for every $h\in C(\partial_\infty M)$ provided the sectional curvature is 
bounded from above by a negative constant and the condition (R) holds at every point $x\in\partial_\infty M$. Clearly, the condition (R) implies 
both the SC condition and the convex conic neighborhood condition.
\end{remark}
\subsubsection*{Acknowledgements}
The authors would like to thank the International Centre of Theoretical Physics (ICTP), Trieste, Italy,  where part of this work has been done, for its 
support and kind hospitality. We are grateful to  Urs Lang for his valuable comments on previous versions of the paper, in particular, for pointing out an 
incorrect assumption in Theorem~\ref {TSMC} in the first version.

\section{Construction of the extension}\label{sec:constext}
In this section we will construct smooth extensions of certain ''angular" functions $h\colon \partial_\infty M\to\R$ associated to boundary points 
$x_0\in \partial_\infty M$ so that the gradients and the Hessians
of the extended functions are controlled in terms of curvature bounds. These functions will be later used in the proof of the SC condition.
Throughout this section we assume that sectional curvatures of $M$ are bounded
both from above and below by
\begin{equation}
\label{curv_assump}-(b\circ\rho)^{2}(x)\le K(P)\le-(a\circ\rho
)^{2}(x)
\end{equation}
for all $2$-dimensional subspaces $P\subset T_{x}M$ and for all $x$ in an appropriate subset of $M$. 
Here $\rho=\rho^o$ stands for the distance function $\rho(x)=d(x,o)$ to a fixed point $o\in M$. Furthermore, $a$
and $b$ are smooth functions $[0,\infty)\to[0,\infty)$ that are constant in
some neighborhood of $0$, $b\ge a$, and they are subject to certain growth conditions that we will specify in due course.
The curvature bounds are needed to control first and second order derivatives
of the ''angular" functions $h$. 
To this end, if $k\colon[0,\infty)\to[0,\infty)$ is a smooth function, we
denote by $f_{k}\in C^{\infty}\bigl([0,\infty)\bigr)$ the solution to the
initial value problem
\begin{equation}
\label{Jacobi_eq}\left\{
\begin{aligned} f_k(0)&=0, \\ f_k'(0)&=1, \\ f_k''&=k^2f_k. \end{aligned} \right.
\end{equation}
It follows that the solution $f_{k}$ is a smooth, non-negative, and increasing function.
We will use extensively various estimates obtained in \cite{HoVa} (and 
originated in the unpublished licentiate thesis \cite{Va_lic}). Therefore
for readers' convenience we use basically the same notation
as in \cite{HoVa}. Thus we let $M$ be a Cartan-Hadamard manifold, $\partial_\infty M$
the sphere at infinity, and $\overline M=M\cup \partial_\infty M$. Recall that the sphere
at infinity is defined as the set of all equivalence classes of unit speed
geodesic rays in $M$; two such rays $\gamma_{1}$ and $\gamma_{2}$ are
equivalent if $\sup_{t\ge0}d\bigl(\gamma_{1}(t),\gamma_{2}(t)\bigr)< \infty$.
For each $x\in M$ and $y\in\overline M\setminus\{x\}$ there exists a unique unit
speed geodesic $\gamma^{x,y}\colon\mathbb{R}\to M$ such that $\gamma^{x,y}%
_{0}=x$ and $\gamma^{x,y}_{t}=y$ for some $t\in(0,\infty]$. If $v\in
T_{x}M\setminus\{0\}$, $\alpha>0$, and $r>0$, we define a cone
\[
C(v,\alpha)=\{y\in\overline M\setminus\{x\}:\sphericalangle(v,\dot\gamma^{x,y}%
_{0})<\alpha\}
\]
and a truncated cone
\[
T(v,\alpha,r)=C(v,\alpha)\setminus\overline B(x,r),
\]
where $\sphericalangle(v,\dot\gamma^{x,y}_{0})$ is the angle between vectors
$v$ and $\dot\gamma^{x,y}_{0}$ in $T_{x} M$. All cones and open balls in $M$
form a basis for the cone topology on $\overline M$.

We start with the following consequence of the Rauch comparison theorem; see e.g. \cite[Lemma 2.6]{HoVa} for a proof.
\begin{lemma}\label{kulma_arvio}
Let $x_0\in M\pois\{o\}$,
$U=M\pois\gamma^{o,x_0}(\R)$, and define $\theta:U\to[0,\pi]$, 
$\theta(x)=\ang_o(x_0,x)$.
Let $x\in U$ and $\gamma=\gamma^{o,x}$.
Suppose that
\[K_M(P)\le -a(t)^2
  \]
for every $t>0$ and for every $2$-dimensional subspace $P\subset T_{\gamma(t)}M$ that contains the radial vector $\dot\gamma_t$.
Then
\[|\del\theta(x)|\le\frac1{(f_a\circ\rho)(x)}.
  \]
\end{lemma}

For a given boundary point $x_{0}\in \partial_\infty M$, let 
\[
v_{0}=\dot\gamma^{o,x_{0}}_{0}\in S_oM=\{v\in T_oM\colon |v|=1\}
\] 
be the initial (unit) vector of the geodesic ray $\gamma^{o,x_{0}}$ from the fixed point $o\in M$. Furthermore, for any constant 
$L\in(8/\pi,\infty)$, we consider cones
\begin{equation}\label{defcone}
\Omega=C(v_0,1/L)\cap M
\end{equation}
and 
\[k\Omega=C(v_0,k/L)\cap M.
  \]
Suppose that
\[-(b\circ\rho)^2(x)\le K_M(P)\le -(a\circ\rho)^2(x)
  \]
for all $x\in 4\Omega$ and all $2$-dimensional subspaces $P\subset T_xM$. We assume that $b$ is monotonic, $b(0)>0$, and that there exist positive 
constants $T_1, \varepsilon,\tilde{\varepsilon}$, and $C_1$ such that
\begin{align}
  \tag{C1}\label{A1}
  a^2(t)&\geq \dfrac{1+\varepsilon}{ t^2\log t},\\  
  \tag{C2}\label{B1}
  b^2(t)&\geq \dfrac{(\log t)^{2\tilde{\varepsilon}}}{t^2},
  \end{align}
for all $t\ge T_1$ and
\begin{align}
  \tag{C3}\label{A3}
  b(t+1)&\le C_1b(t), \\
  \tag{C4}\label{A4}
  b(t/2)&\le C_1b(t)
  \end{align}
for all $t\ge 0$.
Of course, \eqref{A3} (respectively, \eqref{A4}) trivially holds if $b$ is decreasing (respectively, increasing).
It is also worth pointing that the condition \eqref{B1} is not restrictive since $-b^2(t)$ represents the curvature lower bound. 
We collect all these constants and functions together to a data
\[C=(a,b,T_1, \varepsilon,\tilde{\varepsilon},C_1,n).
  \]
Next we define a 
function $h:\partial_\infty M\to\mathbb{R}$ and its (crude) extension $\tilde h:\overline M\to\mathbb{R}$ by
\begin{equation}
\label{eq:hoodef}h(y)=\min\bigl(1,L\sphericalangle(v_{0},\dot\gamma^{o,y}_{0})\bigr),\ y\in \partial_\infty M,
\end{equation}
and
\begin{equation}
\label{eq:hoodeftilde}\tilde h(x)=\min\Bigl(1,\max\bigl(2-2\rho
(x),L\sphericalangle(v_{0},\dot\gamma^{o,x}_{0})\bigr)\Bigr),\ x\in\overline{M},
\end{equation}
respectively.
Then $\tilde h\in C(\overline M)$ and $\tilde h|\partial_\infty M=h$. As the final step in the construction of an ''angular" function we smooth out 
$\tilde{h}$ as in \cite[Section 3.1]{HoVa} to get an extension $h\in C^{\infty}(M)\cap C(\overline M)$. Since the assumptions \eqref{A1}-\eqref{A4}
for functions $a$ and $b$ are weaker than those in \cite{HoVa} (in particular, $b$ need not satisfy condition (A5) in \cite{HoVa}), we 
must present another proofs for some statements and lemmas. 
However, for readers' convenience we list all lemmas that are needed and refer to \cite{HoVa} whenever the proofs there apply verbatim in our setting.

\begin{lemma}\cite[Lemma 3.8]{HoVa}\label{jarj_vaihto}
Let $N$ be a Cartan-Hadamard manifold and $f:N\times N\to\R$ a function. 
Suppose that $f(\cdot,y)\in C^\infty(N)$ for all $y\in N$ and that
\[(x,y)\mapsto X_m(X_{m-1}(\dotsm(X_1(f(\cdot,y)))\dotsm))(x)
  \]
is continuous for all smooth vector fields $X_i\in\cal T(N)$ and all $m\ge0$.
Suppose also that
each $x_0\in N$ has a neighborhood $V\ni x_0$ such that the set
\[\bigcup_{x\in V}\supp f(x,\cdot)
  \]
is bounded.
Define $u:N\to\R$,
\[u(x)=\int_N f(x,y)\,dm_N(y).
  \]
Then $u\in C^\infty(N)$ and 
\begin{equation}\label{der1}
  Xu=\int_N X\bigl(f(\cdot,y)\bigr)\,dm_N(y) 
  \end{equation}
for all $X\in TN$, and
\begin{equation}\label{der2}
  D^2u(X,Y)=\int_N D^2\bigl(f(\cdot,y)\bigr)(X,Y)\,dm_N(y).
  \end{equation}
for all $X,Y\in T_xN$, $x\in N$.
\end{lemma}

The following lemma tells us that for given $k>0$ and $x\in M$, the function $b\circ\rho$ behaves essentially as a constant
in the set $\{y\in M:b(\rho(y))d(x,y)\le k\}$. 
Since $-(b\circ\rho)^2$ represents the curvature lower bound in $4\Omega$, this enables the use of
comparison theorems. Since the proof given in \cite{HoVa} relies on an assumption on the function $b$ that need not be valid in the current paper, we will
present a new proof.

\begin{lemma}\cite[Lemma 3.10]{HoVa}\label{vah_apu}
Let $k>0$.
There exists a constant $c_{1,k}=c_{1,k}(C,k)>1$ such that if 
$x,y\in M$ and $b(\rho(y))d(x,y)\le k$,
then
\[\frac1{c_{1,k}}b\bigl(\rho(x)\bigr)\le b\bigl(\rho(y)\bigr)\le c_{1,k}\,b\bigl(\rho(x)\bigr).
  \]
\end{lemma}

\begin{proof}
Let $x,y\in M$ be such that $b(\rho(y))d(x,y)\le k$. We may assume that $b$ is decreasing since the proof for an increasing $b$ in \cite{HoVa} applies
here, too. Note that the proof for the case of an increasing $b$ is the only point where the assumption \eqref{A3} is used. Suppose first that $\rho(y)\ge \max\bigl\{T_1,\exp\bigl((4k)^{1/\tilde{\varepsilon}}\bigr)\bigr\}$. By \eqref{B1},
\[
b\bigl(\rho(y)\bigr)\ge \frac{4k}{\rho(y)},
\]
and therefore
\[
d(x,y)\le \frac{k}{b\bigl(\rho(y)\bigr)}\le \frac{\rho(y)}{4}\le \frac{1}{4}\bigl(\rho(x)+d(x,y)\bigr).
\]
Consequently, $d(x,y)\le \rho(x)/3$, and so
\[
b\bigl(\rho(y)\bigr)\ge b\bigl(\rho(x)+d(x,y)\bigr)\ge b\bigl(\tfrac{4}{3}\rho(x)\bigr)\ge \frac{1}{C_1}b\bigl(\rho(x)\bigr).  
\]
On the other hand, 
\[
\rho(y)\ge \frac{2}{3}\rho(x),
\]
and therefore
\[
b\bigl(\rho(y)\bigr)\le b\bigl(\tfrac 23 \rho(x)\bigr)\le C_1 b\bigl(\tfrac 43 \rho(x)\bigr)\le C_1 b\bigl(\rho(x)\bigr).
\]
If $\rho(y)\le \max\bigl\{T_1,\exp\bigl((4k)^{1/\tilde{\varepsilon}}\bigr)\bigr\}=:\tilde{c}$, then 
\[
b(\tilde{c})\le b\bigl(\rho(y)\bigr) \le b(0),
\]
and therefore $d(x,y)\le k/b(\tilde{c})$. It follows that $\rho(x)\le \rho(y)+ d(x,y)\le \tilde{c}+ k/b(\tilde{c})$, and hence
\[
b\bigl(\rho(x)\bigr)\ge b\bigl(\tilde{c}+k/b(\tilde{c})\bigr)=c\,b(0)\ge c\, b\bigl(\rho(y)\bigr).
\]
On the other hand,
\[
b\bigl(\rho(x)\bigr)\le b(0)=c\, b(\tilde{c})\le c\, b\bigl(\rho(y)\bigr).
\]
\end{proof}

Fix $\chi\in C^\infty(\R)$ such that $0\le\chi\le 1$, $\supp\chi\subset[-2,2]$, 
and $\chi|[-1,1]\equiv 1$.
After establishing Lemma~\ref{vah_apu} we obtain the following two lemmas exactly as in \cite{HoVa}. 

\begin{lemma}\cite[Lemma 3.11]{HoVa}\label{vol_sopii}
For each $k>0$ there exists a constant $c_2=c_2(C,k)$ such that
\[m_M\Bigl(B\Bigl(x,\frac{k}{b(\rho(x))}\Bigr)\Bigr)\le\frac{c_2}{b(\rho(x))^n}
  \]
holds for all $x\in M$ that satisfy $B\bigl(x,k/b(\rho(x))\bigr)\subset 4\Omega$.
\end{lemma}

\begin{lemma}\cite[Lemma 3.12]{HoVa}\label{tama_ok}
If $\varphi\in C(M)$, then the function 
$f:M\times M\to\R$, 
\[f(x,y)=\chi\bigl(b(\rho(y))d(x,y)\bigr)\varphi(y),
  \]
satisfies the assumptions of Lemma \ref{jarj_vaihto}.
\end{lemma}

Let $\varphi\in C(M)$ and $f$ be as in Lemma \ref{tama_ok} and define
\[\cal R(\varphi)(x)=\int_Mf(x,y)\,dm_M(y).
  \]
Since $\cal R(1)>0$, we can also define $\cal P(\varphi):M\to\R$ by 
\[\cal P(\varphi)=\frac{\cal R(\varphi)}{\cal R(1)}.
  \]
Then $\cal P$ is linear: if $\lambda_1,\lambda_2\in\R$
and $\varphi_1,\varphi_2\in C(M)$, then
\[\cal P(\lambda_1\varphi_1+\lambda_2\varphi_2)=\lambda_1\cal P(\varphi_1)+\lambda_2\cal P(\varphi_2).
  \]
Also, if $k:M\to\R$ is a constant function, then
\[\cal P(k)=k.
  \]
  
The proof of the following lemma under current curvature assumptions requires only a minor change from that in \cite{HoVa}. In fact, we need only a counterpart for the estimate \cite[(3.10)]{HoVa}. 
\begin{lemma}\cite[Lemma 3.13]{HoVa}\label{jatko_lause}
Suppose that $\varphi\in C(\overline M)$.
Extend the function $\cal P(\varphi):M\to\R$ to a function $\overline M\to\R$
by setting
\[\cal P(\varphi)(\bar x)=\varphi(\bar x)
  \]
whenever $\bar x\in \partial_\infty M$.
Then the extended function satisfies $\cal P(\varphi)\in C^\infty(M)\cap C(\overline M)$.
\end{lemma}

\begin{proof}
It is enough to show continuity at infinity since $\cal P(\varphi)\in C^\infty(M)$ by Lemma \ref{jarj_vaihto}.
Fix $\bar x\in \partial_\infty M$ and $\epsilon>0$.
Since $\varphi$ is continuous at $\bar x$, there exist $\delta\in(0,1)$ and $R>0$ such that
$|\varphi(x)-\varphi(\bar x)|<\epsilon$ for every $x\in T(\dot\gamma{}_0^{o,\bar x},\delta,R)$.
\par
Choose $R'>\max(3R/2,T_1)$ such that 
\[\frac{2c_{1,2}t}{(\log t)^{\tilde{\varepsilon}}}\le\frac{\delta}{3}t
  \]
for all $t\ge R'$, where the constant $c_{1,2}$ is given by Lemma~\ref{vah_apu}. 
Let $x\in M\pois B(o,R')$ and let $y\in M$ be such that $b(\rho(y))d(x,y)\le 2$. 
Lemma \ref{vah_apu} and \eqref{B1} imply
\begin{equation}\label{qq1}
  d(x,y)\le\frac{2}{b(\rho(y))}
  \le\frac{2c_{1,2}}{b(\rho(x))}
  \le\frac{\delta}{3}\rho(x).
  \end{equation}
This is the estimate \cite[(3.10)]{HoVa} and after this the original proof applies verbatim.  
\end{proof}

Again, after establishing Lemma~\ref{vah_apu}, the original proof of the following lemma in \cite{HoVa} applies.
\begin{lemma}\cite[Lemma 3.14]{HoVa}\label{yht_arvio}
Let $\varphi\in C(M)$.
Let $x\in M$ be such that \\
$B(x,2c_{1,2}/b(\rho(x)))\subset 4\Omega$ and let $X\in S_xM$.
Then
\begin{equation}\label{arvio_r01}
  |\cal R(\varphi)(x)|\le c_3 b(\rho(x))^{-n}\sup_{y\in B(x,2c_{1,2}/b(\rho(x)))}|\varphi(y)|,
  \end{equation}
\begin{equation}\label{arvio_r1}
  \bigl|X\bigl(\cal R(\varphi)\bigr)\bigr|\le c_3 b(\rho(x))^{1-n}\sup_{y\in B(x,2c_{1,2}/b(\rho(x)))}|\varphi(y)|,
  \end{equation}
and
\begin{equation}\label{arvio_r2}
  \bigl|D^2\bigl(\cal R(\varphi)\bigr)(X,X)\bigr|\le c_3 b(\rho(x))^{2-n}\sup_{y\in B(x,2c_{1,2}/b(\rho(x)))}|\varphi(y)|.
  \end{equation}
Also,
\begin{equation}\label{arvio_r02}
  \cal R(1)(x)\ge c_3^{-1} b(\rho(x))^{-n}.
  \end{equation}
Here $c_3=c_3(C)$ is a constant.
\end{lemma}

The following lemma is a counterpart of \cite[Lemma 3.16]{HoVa}.
\begin{lemma}\label{uus_l}
For every $t_0>0$ and $0<\lambda<1$ there exists a constant $R_1=R_1(C,t_0,\lambda)$ 
such that the following hold.
\par
(a) 
If $x\in 3\Omega\pois B(o,R_1)$ and $y\in B(x,2c_{1,2}/b(\rho(x)))$, then
\begin{equation}\label{turha_viittaus}
  \ang_o(x,y)\le 
  \begin{cases}
  \dfrac{c_{1,2}}{b\left(\rho(x)\right)f_a\left(\lambda\rho(x)\right)}; &\\&\\
\dfrac{c_{1,2}}{b\left(\rho(x)\right)f_a\left(\rho(x)-t_0\right)} &\text{if, in addition, $b$ is increasing},\\ 
	 \end{cases}
  \end{equation}
and
\[y\in 4\Omega\pois B(o,1).
  \]
\par
(b) If $x\in M\pois\bigl(2\Omega\cup B(o,R_1)\bigr)$, then
$B(x,2c_{1,2}/b(\rho(x)))\subset M\pois\bigl(\Omega\cup B(o,1)\bigr)$.
\end{lemma}

\begin{proof}
Suppose that $x\in M$ and $y\in B(x,2c_{1,2}/b(\rho(x)))$.
Choose $R_1'>\max(2,T_1)$ such that 
\[\frac{2c_{1,2}t}{(\log t)^{\tilde{\varepsilon}}}\le\frac{t}{2L}
  \]
for all $t\ge R_1'$.
Suppose that $\rho(x)\ge R_1'$.
Then as in the proof of Lemma~\ref{jatko_lause}, we have
\[
  d(x,y)\le\frac{2c_{1,2}}{b(\rho(x))}
  \le\frac{\rho(x)}{2L}
  \]
and thus $\rho(y)\ge\rho(x)-d(x,y)\ge \rho(x)/2>1$.
Also, Lemma \ref{kulma_arvio} (applied with $a\equiv0$) gives
\[\ang_o(x,y)\le\frac{d(x,y)}{\rho(x)-d(x,y)}\le\frac{1/(2L)}{1-1/(2L)}<\frac{1}{L}.
  \]
From this we see that if $x\notin 2\Omega\cup B(o,R_1')$, then $y\notin \Omega\cup B(o,1)$.
This proves (b).
Also, if $x\in 3\Omega\pois B(o,R_1')$, then $y\in 4\Omega\pois B(o,1)$.
\par
Suppose now that $x\in 3\Omega$ and $\rho(x)\ge R_1'$.
We know by above that then $B(x,2c_{1,2}/b(\rho(x)))\subset 4\Omega\pois B(o,1)$.
We are left to verify the estimates \eqref{turha_viittaus} in (a). 
Fix $0<\lambda<1$ and let $R_1''\ge R_1'$ be so large that
\[\frac{2c_{1,2}t}{(\log t)^{\tilde{\varepsilon}}}\le (1-\lambda)t
  \]
for all $t\ge R_1''$. Then 
\[
d(x,z)\le\frac{2c_{1,2}}{b(\rho(x))}\le (1-\lambda)\rho(x)
\]
for all $z\in B(x,2c_{1,2}/b(\rho(x)))$ whenever $\rho(x)\ge R_1''$. So, for all $z\in B(x,2c_{1,2}/b(\rho(x)))$, we have 
$\rho(z)\ge \rho(x)-d(x,z)\ge\lambda\rho(x)$, in particular, this holds on the geodesic segment from $x$ to $y$. Hence, by integrating 
$|\nabla\ang_o(x,\cdot)|$ along this segment and using Lemma~\ref{kulma_arvio}, we get
\[\ang_o(x,y)\le\frac{d(x,y)}{f_a\bigl(\lambda \rho(x)\bigr)}
 \le\frac{2c_{1,2}}{b(\rho(x))f_a\bigl(\lambda \rho(x)\bigr)}.
  \]
 If $b$ is increasing, we have
\[
d(x,z)\le\frac{2c_{1,2}}{b(\rho(x))}\le t_0
\]
for all $z\in B(x,2c_{1,2}/b(\rho(x)))$ whenever $\rho(x)$ is large enough, say $\rho(x)\ge R_1\ge R_1''$. 
Hence $\rho(z)\ge \rho(x)-t_0$ for all $z\in B(x,2c_{1,2}/b(\rho(x)))$. As above we obtain
\[\ang_o(x,y)\le\frac{d(x,y)}{f_a\bigl(\rho(x)-t_0\bigr)}
 \le\frac{2c_{1,2}}{b(\rho(x))f_a\bigl(\rho(x)-t_0\bigr)}.
  \]
This shows \eqref{turha_viittaus} and ends the proof.
\end{proof}

We extend $h:\partial_\infty M\to\R$, defined by \eqref{eq:hoodef}, to a function 
$h:\overline M\to\R$ by setting
\[h(x)=\cal P(\tilde h)(x),\quad x\in M,
  \]
where $\tilde h$ is the function given by \eqref{eq:hoodeftilde}.
We know by Lemma \ref{jatko_lause} that $h\in C^\infty(M)\cap C(\overline M)$.
The following lemma gives the desired estimates for derivatives of $h$.
We refer to \cite{HoVa} for the proofs of these estimates that are based on Lemmas~\ref{yht_arvio} and \ref{uus_l}.
\begin{lemma}\cite[Lemma 3.16]{HoVa}\label{arvio_lause} (Main lemma)
The extended function $h\in C^\infty(M)\cap C(\overline M)$ 
satisfies 
\begin{equation}\label{arvio1}
  \begin{split}
  |\del h(x)|&\le \begin{cases}
    \dfrac{c_{4}}{f_a\left(\lambda\rho(x)\right)};\\ &\\
  \dfrac{c_{4}}{f_a\left(\rho(x)-t_0\right)} &\text{if, in addition, $b$ is increasing},
  \end{cases}  \\
  \|D^2h(x)\|&\le \begin{cases} 
   \dfrac{c_4 b\left(\rho(x)\right)}{f_a\left(\lambda\rho(x)\right)},\\ &\\
  \dfrac{c_4 b\left(\rho(x)\right)}{f_a\left(\rho(x)-t_0\right)} &\text{if, in addition, $b$ is increasing},
	\end{cases}\\
  \end{split}
  \end{equation}
for all $x\in 3\Omega\pois B(o,R_1)$. 
In addition,
\[h(x)=1
  \]
for every $x\in M\pois\bigl(2\Omega\cup B(o,R_1)\bigr)$. Here $R_1=R_1(C,t_0,\lambda)$ is the constant in Lemma \ref{uus_l}
and $c_4=c_4(C)$ is a constant.
\end{lemma}
Thanks to the previous lemma, we are in position to prove that the SC condition holds under our curvature assumptions.
\section{SC condition}\label{sec:sc}
The aim of this section is to give a proof of Theorem \ref{thmSC}. We consider two cases depending on the assumption made on the sectional curvature. 

\subsection{Proof of Theorem \ref{thmSC} under assumption (1).}
Throughout this subsection we assume that functions $a$ and $b$ satisfy the conditions \eqref{A1}-\eqref{A4} from Section~\ref{sec:constext}. 
Recall the following (Jacobi field) estimates:
\begin{proposition}\cite[Prop. 3.4]{choi}
\label{jacest}
Suppose that 
\[
a^2(t)\geq \frac{1+\varepsilon}{t^2\log t}, 
\]
for some $\varepsilon>0$ on $[R_0,\infty)$. Then, for any $0<\varepsilon_1<\varepsilon$, 
there exists $R_1\geq R_0$ such that, for all $t\geq R_1$,
\[
f_a(t)\geq t (\log t)^{1+\varepsilon_1},\quad 
\frac{f_a^\prime (t)}{f_a(t)}\geq \frac{1}{t}+\frac{1+\varepsilon_1}{t\log t}.
\]
\end{proposition}
Note that in the above setting we have an estimate
\begin{equation}\label{f'est}
f_a^\prime (t)\ge (\log t)^{1+\varepsilon_1} + (1+\varepsilon_1)(\log t)^{\varepsilon_1}
\end{equation} 
for $t\ge R_1$.
Let us first consider the case 
\[
a^2(t)=\dfrac{1+\varepsilon}{ t^2\log t}.
\] 

We prove the following:

\begin{proposition}\label{1stprop}
Assume that 
\[
a^2(t)=\dfrac{1+\varepsilon}{ t^2\log t}\ \text{ and }\ 
b^2(t)=\dfrac{(\log t)^{2\tilde{\varepsilon}}}{t^2},
\] 
with $\varepsilon>\tilde{\varepsilon}>0$. Fix $\varepsilon_1\in (\tilde{\varepsilon},\varepsilon)$ and 
define $g(t)=(\log t)^{\alpha}$ for $t>1$ and for a constant $\alpha\in (0,\varepsilon_1-\tilde{\varepsilon})$. 
Then there exist $R>0$ and $\beta>0$ such that, for all $x_0\in\partial_{\infty}M$ and $0<\beta_0\leq \beta$, the 
set 
\[
\cC=\left\{x\in M \colon \varphi\bigl(h(x)\bigr) g\bigl(\rho(x)\bigr)\leq R\right\}
\] 
is strictly convex, where $h$ is the function defined in Section~\ref{sec:constext} related to $x_0\in\partial_{\infty}M$ and 
$\varphi:[0,\infty)\to [0,1]$ is a smooth function such that $\varphi(0)=1$, $\varphi|[\beta_0,\infty)=0$, and 
$|\varphi^\prime(t)|, |\varphi^{\prime \prime}(t)|\leq C$ for all $t\in [0,\beta_0]$. In particular, $M$ satisfies 
the SC condition.
\end{proposition}

\begin{remark}
We begin by noticing that $\lim_{x\to x_0}\varphi(h(x))>0$, and consequently $\lim_{x\to x_0}\varphi\bigl(h(x)\bigr)g\bigl(\rho(x)\bigr)=\infty$,
since $h(x_0)=0$ and $h$ is continuous in $\overline{M}$. On the other hand, $h(x)=1$ if $\ang(\dot\gamma^{o,x_0},\dot\gamma^{o,x}_0)\ge 2/L$
and $\rho(x)$ is large enough. Thus, given a relatively open subset $W\subset\partial_{\infty}M$ containing $x_0$, we ensure that
$x_0\in\operatorname*{Int}\left(  \partial_{\infty}(M\setminus\cC)\right)  \subset W$ by choosing the constants $L$ (see \eqref{defcone}, 
\eqref{eq:hoodef}) and $\beta_0$ properly.
Therefore, if we prove that $\cC$  is strictly convex, $M$ will satisfy the SC condition.  \\
\end{remark}

 \begin{proof}
We will use Lemma~\ref{arvio_lause} with $\lambda=3/4$; any other choice of $0<\lambda<1$ would work as well.
Throughout the proof, all computations are done on the boundary $\partial\cC$ where 
$\varphi\bigl(h(x)\bigr) g\bigl(\rho(x)\bigr)=R$, with sufficiently large $R$. 
To prove the strict convexity of $\cC$, we will show that $D^2 (\varphi(h) g(\rho))$ is positive definite for an arbitrary unit vector field $X$ 
on $\partial\cC$ that is tangential to $\partial\cC$, i.e. 
\[
X\bot \left(g(\rho )\varphi^\prime (h)\nabla h+\varphi(h) g^\prime (\rho)\nabla \rho\right),
\]
in other words
\begin{equation}\label{relortsc2}
\langle X,\nabla\rho\rangle=-\frac{\varphi^\prime (h)g(\rho)\langle X,\nabla h\rangle}{\varphi(h)g^\prime (\rho)}
=-\frac{\varphi^\prime (h)g^2(\rho)\langle X,\nabla h\rangle}{R g^\prime (\rho)}.
\end{equation}
A direct computation gives
\begin{align}
\label{estsce1}
D^2 (\varphi(h) g(\rho))(X,X)&= 
g(\rho)\varphi^{\prime \prime }(h)\langle X,\nabla h\rangle^2 + g(\rho)\varphi^\prime (h)D^2 h (X,X) \nonumber \\
&\quad +\varphi(h) g^{\prime \prime} (\rho)\langle X,\nabla\rho\rangle^2 + \varphi(h)g^\prime (\rho) D^2 \rho(X,X)\\
&\quad +2\varphi^\prime (h) g^\prime (\rho)\langle X,\nabla h\rangle\langle X,\nabla\rho\rangle.\nonumber 
\end{align}
On $\partial\cC$ we have $\varphi(h)=R/g(\rho)$, and therefore we obtain from \eqref{relortsc2} and \eqref{estsce1}
\begin{align}
\label{estsce2}
D^2 (\varphi(h) & g(\rho))(X,X)= 
\frac{R g^\prime (\rho) D^2 \rho(X,X)}{g(\rho)} + g(\rho)\varphi^\prime (h)D^2 h (X,X)  \\
&+\left( g(\rho)\varphi^{\prime \prime }(h) + 
 \frac{g^{\prime \prime}(\rho)\varphi^\prime (h)^2 g^3(\rho)}{R g^\prime(\rho)^2}
 -\frac{2\varphi^\prime(h)^2 g^2 (\rho)}{R}
\right)\langle X,\nabla h\rangle^2.\nonumber
\end{align}
Lemma~\ref{arvio_lause} and \eqref{relortsc2} imply that 
\begin{align*}
\langle X,\nabla\rho\rangle^2 & =\left(\frac{\varphi^\prime (h) g^2(\rho)\langle X,\nabla h\rangle}{R g^\prime (\rho)}\right)^2
\le \left(\frac{c_4\varphi^\prime (h) g^2(\rho)}{R g^\prime (\rho)f_a\bigl(\tfrac34 \rho\bigr)}\right)^2\\
&\le \frac{c}{R^2\rho^2(\log\rho)^{2(\varepsilon_1-\alpha)}}\le \frac{1}{2}.
\end{align*}
Using the Hessian comparison theorem, we deduce from the previous estimate that
\begin{align}
\label{estsce3}
D^2\rho(X,X) &\geq \dfrac{f_a^\prime (\rho)}{f_a (\rho)} \left( \norm{X}^2 -\langle X,\nabla \rho \rangle^2\right)\nonumber\\
& \geq \dfrac{f_a^\prime (\rho)}{f_a (\rho)} \left(1-\left(\dfrac{c_4 \varphi^\prime (h) g(\rho )^2 }{R g^\prime (\rho) 
f_a\bigl(\tfrac34 \rho\bigr)} \right)^2 \right)\\
&\ge \dfrac{f_a^\prime (\rho)}{2 f_a (\rho)}.\nonumber
\end{align}
Next we estimate from below the terms on the right side of \eqref{estsce2}. 
The first term, which will dominate the others, can be estimated by using Proposition~\ref{jacest} and \eqref{estsce3} as
\begin{align}\label{1stest}
\frac{R g^\prime (\rho) D^2 \rho(X,X)}{g(\rho)}&=\frac{R\alpha D^2\rho(X,X)}{\rho\log\rho}\nonumber\\
 &\ge  \frac{R\alpha}{2\rho\log\rho}\left(\frac{1}{\rho}+\frac{1+\varepsilon_1}{\rho\log\rho}\right)\\
 &\ge \frac{R\alpha}{2\rho^2\log\rho}.\nonumber
\end{align}
The second term has a lower bound
\begin{equation}\label{2ndest}
g(\rho)\varphi^\prime(h)D^2 h(X,X)\ge -\frac{c}{\rho^2(\log\rho)^{1+\varepsilon_1 -\tilde{\varepsilon} -\alpha}}
\end{equation}
by Lemma~\ref{arvio_lause} and Proposition~\ref{jacest}. For the last three terms we first have estimates
\begin{align*}
g(\rho)\varphi^{\prime \prime }(h) & \ge -c(\log\rho)^{\alpha},\\
\frac{g^{\prime \prime}(\rho)\varphi^\prime (h)^2 g^3(\rho)}{R g^\prime(\rho)^2} & \ge \frac{-c(\log\rho)^{1+2\alpha}}{R}
\ge -c(\log\rho)^{1+\alpha},\\
-\frac{2\varphi^\prime(h)^2 g^2 (\rho)}{R} & \ge 
 -\frac{c(\log\rho)^{2\alpha}}{R} \ge -c(\log\rho)^{\alpha},
\end{align*}
where we used the fact that $g(\rho)=(\log\rho)^{\alpha}\ge R$ on $\partial\cC$.
Combining these estimates with Lemma~\ref{arvio_lause} and Proposition~\ref{jacest} gives
\begin{align}\label{3rdest}
&\left( g(\rho)\varphi^{\prime \prime }(h) + 
 \frac{g^{\prime \prime}(\rho)\varphi^\prime (h)^2 g^3(\rho)}{R g^\prime(\rho)^2}
 -\frac{2\varphi^\prime(h)^2 g^2 (\rho)}{R}
\right)\langle X,\nabla h\rangle^2\\
&\quad  \ge \frac{-c(\log\rho)^{1+\alpha}}{f_a^2\bigl(\tfrac34 \rho\bigr)}
\ge \frac{-c}{\rho^2(\log\rho)^{1+2\varepsilon_1-\alpha}}.\nonumber
\end{align}
Plugging in estimates \eqref{1stest}, \eqref{2ndest}, and \eqref{3rdest} to \eqref{estsce2} and recalling that $\alpha<\varepsilon_1-\tilde{\varepsilon}$ we obtain
\[
D^2 (\varphi(h) g(\rho))(X,X) \ge 
\frac{R\alpha}{2\rho^2\log\rho} -\frac{c}{\rho^2(\log\rho)^{1+\varepsilon_1 -\tilde{\varepsilon} -\alpha}}
\frac{-c}{\rho^2(\log\rho)^{1+2\varepsilon_1-\alpha}}
>0
\]
for $R$, and hence $\rho$, large enough.
This establishes the proposition.
\end{proof}
More generally, we have :
\begin{proposition}\label{3ndprop}
Assume that 
\[
a^2(t)\geq\dfrac{1+\varepsilon}{ t^2\log t},
\]
$b(t+1)\le C_1b(t),\ b(t/2)\le C_1 b(t)$, and 
\[
\frac{(\log t)^{\tilde{\varepsilon}}}{t}\le b(t) \le 
\begin{cases}
\dfrac{f_a^\prime (t)}{t (\log t)^{1+2\alpha}}\dfrac{f_a (\lambda t)}{f_a (t)};&\\&\\
\dfrac{f_a^\prime (t)}{t (\log t)^{1+2\alpha}}\dfrac{f_a (t-t_0)}{f_a (t)},\text{ if, in addition, $b$ is increasing},
\end{cases}
\] 
for $t$ large enough and for some $\varepsilon>\tilde{\varepsilon}>0,\ t_0>0,\ 0<\lambda<1$, and $\alpha\in (0,\varepsilon_1-\tilde{\varepsilon})$,
with $\varepsilon_1\in (\tilde{\varepsilon},\varepsilon)$. 
Then there exist $R>0$ and $\beta$ such that, for all $x_0\in\partial_{\infty}M$ and $\beta_0\leq \beta$, the set 
\[
\cC=\left\{x\in M \colon \varphi\bigl(h(x)\bigr) g\bigl(\rho(x)\bigr)\leq R\right\}
\]
is strictly convex, where the functions $\varphi$ and $g(t)=(\log t)^{\alpha}$ are as in Proposition~\ref{1stprop}
and $h$ is the function defined in Section~\ref{sec:constext} related to $x_0\in\partial_{\infty}M$. 
In particular, $M$ satisfies the SC condition.
\end{proposition}

\begin{proof}
The proof of Proposition \ref{1stprop} applies almost verbatim. We will only consider the general case for $b$, the case of an increasing $b$
being a simple adaptation. Proceeding as previously, we get
\[
\langle X,\nabla\rho\rangle^2\le \frac{1}{2}
\]
and
\[
\frac{R g^\prime (\rho) D^2 \rho(X,X)}{g(\rho)}\ge \frac{R\alpha}{\rho\log\rho}\dfrac{f_a^\prime (\rho)}{2 f_a (\rho)}
\]
for $R$, and hence $\rho$, sufficiently large. As in \eqref{3rdest}, we have
\[
\left( g(\rho)\varphi^{\prime \prime }(h) + 
 \frac{g^{\prime \prime}(\rho)\varphi^\prime (h)^2 g^3(\rho)}{R g^\prime(\rho)^2}
 -\frac{2\varphi^\prime(h)^2 g^2 (\rho)}{R}
\right)\langle X,\nabla h\rangle^2\quad  \ge \frac{-c(\log\rho)^{1+\alpha}}{f_a^2\bigl(\lambda \rho\bigr)}.
\]
By Proposition~\ref{jacest}, we obtain
\[
\frac{R\alpha}{4\rho\log\rho}\dfrac{f_a^\prime (\rho)}{f_a (\rho)}\ge \frac{R}{4\rho^2\log\rho}
\]
and
\[
\frac{c(\log\rho)^{1+\alpha}}{f_a^2\bigl(\lambda \rho\bigr)}\le 
\frac{c(\log\rho)^{1+\alpha}}{\lambda^2 \rho^2(\log\rho)^{2+2\varepsilon_1}}
\le \frac{R}{4\rho^2\log\rho},
\]
and therefore
\[
\frac{R\alpha}{\rho\log\rho}\dfrac{f_a^\prime (\rho)}{2 f_a (\rho)} -
 \frac{c(\log\rho)^{1+\alpha}}{f_a^2\bigl(\lambda \rho\bigr)}
\ge \frac{R\alpha}{4\rho\log\rho}\dfrac{f_a^\prime (\rho)}{f_a (\rho)}
\]
for $R$, consequently $\rho$, large enough.
Plugging these estimates into \eqref{estsce2} we obtain 
\begin{align*}
D^2 (\varphi(h) g(\rho))(X,X) 
&\geq - g(\rho)|f^\prime (h)| |D^2 h (X,X)|+ \frac{R\alpha}{4\rho\log\rho} \dfrac{f_a^\prime (\rho)}{f_a(\rho)}.
\end{align*}
Using \eqref{arvio1} and the assumption 
\[
\dfrac{b(t)}{f_a (\lambda t)} \leq \dfrac{f_a^\prime (t)}{t (\log t)^{1+2\alpha}f_a (t)}, 
\] 
we have
\[
g(\rho)|\varphi^\prime (h)| |D^2 h (X,X)|\leq   \dfrac{c(\log \rho)^\alpha  b(\rho)}{f_a (\lambda\rho)}\leq \dfrac{cf_a^\prime (\rho)}{\rho (\log \rho)^{1+\alpha} f_a(\rho)} < \frac{R\alpha}{4\rho\log\rho} \dfrac{f_a^\prime (\rho)}{f_a(\rho)}
\]
for $R$, and hence $\rho$, sufficiently large.
This concludes the proof.
\end{proof}

\subsection{Proof of Theorem \ref{thmSC} under assumption (2)}
In this subsection we assume that
\[
-b^{2} (\rho (x)) \le K_M(x) \le -a^2(\rho (x))
\]
for all $x\in M$, with $\rho (x)=d(x,o)\geq R^\ast $, where $a$ and $b$ are non-decreasing continuous functions such that $k:=b(0)>0$ and
\[
\underset{t\rightarrow +\infty}{\lim} \dfrac{t (\log t)^{1+\varepsilon}f_a( t-2) b(t)}{f_a^\prime (t-2) f_a (t-3)}<\infty .
\]

Let $\phi:[0,+\infty)\rightarrow [0,1]$ be a smooth non-decreasing function such that $\phi([0,1/2])=0$ and $\phi\equiv1$ on $[1,+\infty)$. 
We also assume that there exists a constant $L>0$ such that $\phi^{\prime},\phi^{\prime\prime}\leq L$. 
We are now in position to state the following lemma which is a crucial tool in our construction. In what follows,
$S(p,r)=\partial B(p,r)$.
\begin{lemma} 
\label{betaL}
There exists a constant $\beta>0$ depending only on $k$ and $L$ such that for
\begin{equation}
\varepsilon_{R}:=\dfrac{\beta}{ b(R+1)}\dfrac{f_a^\prime (R)}{f_a(R)},\label{epsilonR}
\end{equation}
the sublevel set 
\[
 C_{R,p}=\{x\in M\colon  \left(\rho -\varepsilon_{R}\phi(\rho_p)\right)(x)\leq R\}
\]
is a strictly convex set for all  $R\ge R^{\ast}$ and $p\in S(o,R)$.
\end{lemma}

\begin{proof}
The strict convexity of the set $C_{R,p}$ will follow from the fact that 
\[
D^{2}(\rho -\varepsilon_{R}\phi\circ\rho_p)(X,X)>0
\]
whenever $X$ is a unit vector field tangent to the level set of $C_{R,p}$ at $q\in \partial C_{R,p}$, i.e.   
\begin{equation}\label{ort}
\langle X, \nabla\rho-\varepsilon_{R}\nabla(\phi \circ \rho_p)\rangle=0.
\end{equation} 
Let $X$ be such a vector field. We begin by noticing, since $\rho$ is a distance function, that 
$$D^{2}\rho(X,X)=D^{2}\rho(X^{\perp},X^{\perp}),$$ where $X^{\perp}=X-\langle
X,\nabla\rho\rangle\nabla\rho$.
Straightforward computations then yield to
\begin{align}
\label{betaLe1}
D^{2}(\rho-\varepsilon_{R}\phi\circ\rho_p)(X,X)   
 & =D^{2}\rho(X^\perp,X^\perp)-\varepsilon_{R}\phi^{\prime}D^{2}\rho_{p}
(X,X)\nonumber\\
&-\varepsilon_{R}\phi^{\prime\prime}\langle\nabla\rho_{p},X\rangle^{2}.
\end{align}
Let us estimate each terms of \eqref{betaLe1}. Since $K_{M}(x)\leq - a^2(\rho (x))$, for all $x\in \partial C_{R,p}$ and $X$ satisfies \eqref{ort}, it follows from the Hessian comparison theorem that

\begin{align}
D^{2}\rho(X^\perp,X^\perp) & \geq \dfrac{f_a^\prime (x)}{f_a(x)} \left(\| X\|^{2}-\langle X,\nabla\rho\rangle^{2}\right) \nonumber\\
& = \dfrac{f_a^\prime (x)}{f_a(x)} \left(1 - \varepsilon_{R} \phi^{\prime}\langle X,
\nabla\rho_{p} \rangle \langle X, \nabla\rho \rangle\right) \nonumber \\
&  \ge
\dfrac{f_a^\prime (R)}{f_a(R)} (1-\varepsilon_{R} L) .\label{D2rhoo}
\end{align}

Since $\phi^{\prime}\equiv0$ on  $[0,1/2]\cup [1,+\infty)$, $b_{R+1}=b(R+1)\geq k$ and $K_{M}\geq-b^2_{R+1}$ in $B(p,1)$, applying the Hessian comparison theorem in the other direction, we deduce that
\begin{align}
\phi^{\prime}D^{2}\rho_{p}(X,X)  &  \leq\phi^{\prime}b_{R+1}\coth(b_{R+1}\rho
_{p})\Vert X-\langle X,\nabla\rho_{p}\rangle\nabla\rho_{p}\Vert^{2}\nonumber\\
&  =\phi^{\prime}b_{R+1}\coth(b_{R+1}\rho_{p})\left(  1-\langle
X,\nabla \rho_{p}\rangle^{2}\right) \nonumber\\
&\leq L b_{R+1}\coth(k/2). \label{D2rhop}
\end{align}
Finally, plugging \eqref{D2rhoo} and \eqref{D2rhop} into \eqref{betaLe1}, we get that
\[
D^{2}(\rho-\varepsilon_{R}f_{p})(X,X)    \geq \dfrac{f_a^\prime (R)}{f_a(R)} -\varepsilon
_{R}L\left(\dfrac{f_a^\prime (R)}{f_a(R)} + b_{R+1}\coth(k/2)+1\right).
\]
Now, taking 
\[
\varepsilon_R = \dfrac{\beta}{b_{R+1}}\dfrac{f_a^\prime (R)}{f_a(R)}
\] 
for some $\beta>0$ small enough, it is easy to see that
\[
D^{2}(\rho-\varepsilon_{R}\varphi\circ \rho_p)(X,X)> 0.
\]
This concludes the proof of the lemma.
\end{proof}
\begin{remark}
\label{rmk1}
The value of $\varepsilon_R$ in the previous lemma will be crucial in the sequel. We would like to point out that it is possible to replace the term $b(R+1)$ appearing in the definition of $\varepsilon_R$ by some terms involving only $b(R+\varepsilon)$ for some arbitrarily small positive $\varepsilon$. 
In order to prove this fact, we will modify the cut-off function $\phi$. More precisely, we let $\phi:[0,+\infty)\rightarrow [0,1]$ be a smooth non-decreasing function such that
$\phi([0,\varepsilon])=0$, $\phi\equiv1$ on $[2\varepsilon,+\infty)$, where $\varepsilon>0$ will be determined in the sequel. We also assume that there exists a constant $L>0$ such that $\phi^{\prime}\leq \dfrac{L}{\varepsilon}$ and $\phi^{\prime\prime}\leq \dfrac{L}{\varepsilon^2}$. 
Since $\phi^{\prime}\equiv0$ on  $[0,\varepsilon]\cup [2\varepsilon,+\infty)$ and $K_{M}\geq-b^2_{R+2\varepsilon}$ in $B(p,2\varepsilon)$, applying the Hessian comparison theorem in the other direction, we deduce that
\begin{align}
\phi^{\prime}D^{2}\rho_{p}(X,X)  &  \leq\phi^{\prime}b_{R+2\varepsilon}\coth(b_{R+2\varepsilon}\rho
_{p})\Vert X-\langle X,\nabla\rho_{p}\rangle\nabla\rho_{p}\Vert^{2}\nonumber\\
&  =\phi^{\prime}b_{R+2\varepsilon}\coth(b_{R+2\varepsilon}\rho_{p})\left(  1-\langle
X,\nabla \rho_{p}\rangle^{2}\right) \nonumber\\
&\leq \dfrac{L}{\varepsilon} b_{R+2\varepsilon}\coth(b_{R+2\varepsilon} \varepsilon). \label{D2rhopr}
\end{align}
Finally, plugging \eqref{D2rhoo} and \eqref{D2rhopr} into \eqref{betaLe1}, we get that
\[
D^{2}(\rho-\varepsilon_{R}f_{p})(X,X)    \geq \dfrac{f_a^\prime (R)}{f_a(R)} -\varepsilon
_{R}L\left(\dfrac{f_a^\prime (R)}{f_a(R)} + \dfrac{b_{R+2\varepsilon}}{\varepsilon}\coth(b_{R+2\varepsilon} \varepsilon )+\dfrac{1}{\varepsilon^2}\right).
\]
Now, taking 
\[
\varepsilon_R =\beta \dfrac{f_a^\prime (R)}{f_a(R)} \min \left\{ \varepsilon^2, \dfrac{\varepsilon}{b_{R+2\varepsilon} \coth(b_{R+2\varepsilon} \varepsilon )}   \right\}
\] 
for some $\beta>0$ small enough, it is easy to see that
\[
D^{2}(\rho-\varepsilon_{R}\varphi\circ \rho_p)(X,X)> 0.
\]

\end{remark}

\begin{proof}[Proof of Theorem \ref{thmSC} under assumption \eqref{pinch}]
To prove that $M$ satisfies the SC condition, it is sufficient to show that given $\gamma(\infty)\in\partial_{\infty} M$ for some geodesic $\gamma$ such that $\gamma (0)=o$ and $0<\alpha<\pi/2$, there exists
a convex set $C\subset M$ such that

\begin{enumerate}
\item[(i)] $\partial_{\infty}C\supseteq
\{\tilde{\gamma}(\infty)\colon \sphericalangle_{o}(\tilde{\gamma}(\infty),\gamma(\infty))\geq
2\alpha\}$

\item[(ii)] $C\cap T(\dot{\gamma}(0),\alpha,r_{0}+1)=\emptyset$
for some $r_{0}\geq R^{\ast}$ large enough.
\end{enumerate}

We follow exactly the same construction as the one  in \cite{RT} (see also \cite{andJDG} and \cite{borbpams}). We outline it for the sake of self-containment. This construction is based on an iterative procedure so let us begin by initializing it. Let $r_{0}\geq R^{\ast}$ to be determined in the sequel and let $v:=\dot{\gamma}_0\in T_{o}M$. We start by defining the following
sets:
\[
C_{0}:=B(o,r_{0}),\,\,\,\,\,\,T_{0}:=\{p\in S(o,r_{0})\colon \sphericalangle
(\gamma^{o,p},v))<\alpha\},\,\,\,\,\,\,D_{0}=\emptyset .
\]
Using Lemma \ref{betaL}, we can find an uniform $\varepsilon_{0}
:=\varepsilon_{r_{0}}$ such that for each $p\in S(o,r_{0})$, the sublevel
set
\[
C_{0,p}:=\{x\in M\colon\text{ }(\rho_{o}-\varepsilon_{0}f_{p})(x)\leq r_{0}\}
\]
is convex. Then, we define a new collection of sets:
\begin{align*}
&  \widetilde{C_{1}}:=\bigcap_{p\in T_{0}}C_{0,p},\,\,\,\,\,\,C_{1}
:=\widetilde{C_{1}}\setminus D_{0},\,\,\,\,\,\,r_{1}:=r_{0}+\varepsilon_{0}\\
&  D_{1}:=B(o,r_{1})\setminus C_{1},\,\,\,\,\,\,T_{1}:=\overline{S(o,r_{1})\setminus\partial C_{1}}.
\end{align*}
The induction scheme is then the following : Let us assume that $C_{k}
,r_{k},D_{k},T_{k}$ are defined for all $k\leq n-1$, $n\geq1$. Using Lemma \ref{betaL}, we get the existence of an uniform
$\varepsilon_{n}$ such that if $p\in S(o,r_{n-1})$, the set $C_{n,p}
:=\{\rho_{o}-\varepsilon_{n-1}f_{p}\leq r_{n-1}\}$ is convex. The inductive construction is then given by
\begin{align*}
&  \widetilde{C_{n}}:=\bigcap_{p\in T_{n-1}}C_{n-1,p},\,\,\,\,\,\,C_{n}
:=\widetilde{C_{n}}\setminus D_{n-1},\,\,\,\,\,\,r_{n}:=r_{n-1}+\varepsilon
_{n-1}\\
&  D_{n}:=B(o,r_{n})\setminus C_{n},\,\,\,\,\,\,T_{n}:=\overline{S(o,r_{n})\setminus\partial C_{n}}.
\end{align*}
Finally, we set 
\[
C:=\bigcup_{n\ge0}C_{n}\quad\text{and}\quad 
D:=\bigcup_{n\ge0} D_{n}.
\]
In exactly the same way as in \cite{RT}, one can prove that $C$ is a convex set, $M=C\cup D$ and that $(ii)$ holds. Therefore, we are reduced to prove $(i)$. We first notice that, if $x\in D$, then the following estimate holds
\[
\sphericalangle_{o}(v,\dot{\gamma}^{o,x}(0))\leq\alpha+\sum_{n=0}^{+\infty
}\theta_{n},
\]
where $\theta_n$ is defined by
\[
 \theta_{n}:=\sup\{\sphericalangle(\dot{\gamma}^{o,p}(0),\dot{\gamma}^{o,q}(0))\colon p\in T_{n},q\in S(p,1)\}.
 \]
We claim that, taking $r_0$ large enough, 
\[
\sum_{n=0}^{+\infty}\theta_{n} \leq \alpha.
\] 
We begin by rewriting this sum as 
\[
\displaystyle\sum_{i=0}^{\infty}\theta_{i}   =\sum_{n=0}^{\infty}\sum_{r_{i}\in I_{n}}\theta_{i},
\] 
where $I_{n}:=[r_{0}+n,r_{0}+n+1]$. Denote by $t_{n}$ the number of
$r_{i}$'s on the interval $I_{n}$, $n\geq0$. To estimate $t_n$ let us define $j_{n}$
as the index of the greatest $r_{i}$ on $I_{n}$. Since $\varepsilon_{i}$ is
decreasing on $i$, we have:
\begin{align*}
t_{n}\varepsilon_{j_{n}}  &  \leq\varepsilon_{j_{n}-1}+\varepsilon_{j_{n}
-2}+\cdots+\varepsilon_{j_{n}-t_{n}}\\
&  =(r_{j_{n}}-r_{j_{n}-1})+(r_{j_{n}-1}-r_{j_{n}-2})+\cdots+(r_{j_{n}
+1-t_{n}}-r_{j_{n}-t_{n}})\\
&  
\leq(r_{0}+n+1)-(r_{0}+n)=1.
\end{align*}
This implies that $$t_n \le \varepsilon_{j_n}^{-1} \leq \dfrac{b(r_0 +n+2)}{\beta}\dfrac{f_a (r_0+n)}{f_a^\prime(r_0+n)}.$$
Using Lemma \ref{kulma_arvio} and our assumptions on the curvature, we deduce that
$$\theta_{n}\leq \dfrac{c}{f_a (r_0+n-1)},$$
for some constant $c>0$. Combining the above estimates and since by assumption
\[
\underset{t\rightarrow \infty}{\lim} \dfrac{t (\log t)^{1+\varepsilon}f_a(t-2) b(t)}{f_a (t-3) f_a^\prime (t-2)}<\infty ,
\] 
we obtain that
\begin{align*}
\sum_{i=0}^{\infty}\theta_{i}  &  \leq\sum_{n=0}^{\infty}t_{n}\frac{c}{f_a(r_0+n-1)}\leq \dfrac{c}{\beta}
\sum_{n=0}^{\infty}\frac{b(r_0+n+2)}{f_a(r_0+n-1)}\dfrac{f_a (r_0+n)}{f_a^\prime(r_0+n)}<\infty.\\
\end{align*}
This proves the claim.
\end{proof}
\begin{remark}
\label{rmk2}
We notice that it is possible to improve the previous estimate by taking $I_n=\bigl[r_0+\sum_{i=1}^n \tfrac 1i, r_0 +\sum_{i=1}^{n+1} \tfrac 1i \bigr]$. 
With this choice, we have
\[
t_n\leq \dfrac{1}{n\varepsilon_{j_n}}\leq \dfrac{b\left(1+r_0 +\sum_{i=1}^{n+1}\tfrac 1i \right)}{n\beta}
\dfrac{f_a \left(r_0+\sum_{i=1}^{n}\tfrac 1i \right)}{f_a^\prime\left(r_0+\sum_{i=1}^{n}\tfrac 1i \right)},
\]
and
\[
\theta_{n}\leq \dfrac{c }{f_a \left(r_0+\sum_{i=1}^{n+1}\tfrac 1i \right)}.
\]
This yields to
\begin{align*}
\sum_{i=0}^{\infty}\theta_{i}  &  \leq\sum_{n=0}^{\infty}t_{n}\frac{c}{f_a\left(r_0+\sum_{i=1}^{n+1}\tfrac 1i \right)}
\leq \dfrac{c}{\beta}
\sum_{n=0}^{\infty}\dfrac{ b\left(1+r_0+\sum_{i=1}^{n+1}\tfrac 1i \right)}{n f_a\left(r_0+\sum_{i=1}^{n+1}\tfrac 1i \right)}
\dfrac{f_a \left(r_0+\sum_{i=1}^{n}\tfrac 1i \right)}{f_a^\prime\left(r_0+\sum_{i=1}^{n}\tfrac 1i\right)}.\\
\end{align*}
\end{remark}

\section{Asymptotic Plateau problem}
In this section we prove Theorems~\ref{TSMC}, \ref{TSMC2}, and \ref{TSMC3} on the asymptotic Plateau problem.
Throughout the section we assume that $M$ is an $n$-dimensional Cartan-Hadamard manifold, $n\ge 3$, that satisfies the SC condition.
We start with the following fundamental lemma.
\begin{lemma}\label{funlem} 
For every closed subset $\Gamma\subset\partial_{\infty}M$ there exists a closed subset $\Lambda\subset M$ whose interior
$\Int\Lambda$ is a convex subset of $M$ and $\partial_{\infty}\Lambda=\Gamma$. 
\end{lemma}
\begin{proof}
We extend $\Gamma$ to a closed subset $S$ of $M$ by fixing $o\in M$ and extending $\Gamma$
radially through the exponential map based at $o$. Thus $S$ is the cone over $\Gamma$ from the point $o\in M$. Using the notation
in Section~\ref{sec:constext}
\[
S=\cone_o\Gamma=\{\gamma^{o,x}(t)\colon t\ge 0, x\in\Gamma\}.
\]
We then construct a convex subset
$\Lambda$ of $M$ that contains $S$ in its interior and
$\partial_{\infty}\Lambda=\Gamma$. To this end, we choose an open subset $O$
of $M$ such that $\partial_{\infty}O=\Gamma$, for example, a tubular
neighbourhood of $S$ with fixed radius. Given $x\in \partial_{\infty}M\setminus\Gamma$ take an
open set $W_{x}\subset\partial_{\infty}{M}\setminus\Gamma$ such that $x\in W_{x}$. By the SC condition, there exists a $C^{2}$-smooth open
subset $\Omega_{x}$ of $M$ such that $x\in\ir\partial_{\infty}\Omega_{x}  \subset W_{x}$ and
$M\setminus\Omega_{x}$ is convex. 
Let $d_{x}:M\to [ 0,\infty)$ denote the
Riemannian distance $d_x=\dist(y,M\setminus\Omega_{x})$.
By \cite[Theorem 4.1]{choi}, any two points of $M\setminus \Omega _{x}$ can be connected by a geodesic segment entirely contained in 
$M\setminus \Omega _{x}$. Since $K_{M}\leq 0$, we may then apply the Hessian comparison theorem \cite[(2.49) Theorem]{kasue} (see also \cite{warner})
to immediately obtain $D^2d_{x}(v,v)\geq 0$ for all $y\in \Omega _{x}$ and for all $v\in T_{y}M$, with 
$v\perp \nabla d_{x}(y).$ Since $D^2 d_{x}(v,v)=\left\langle B(v,v),\eta \right\rangle $, where $\eta
=-\nabla d_{x}$ and $B$ is the second fundamental form of the level hypersurface $L_{y}$ of $d_{x}$ through $y$, 
it follows that the set $\{z\in M\colon d_x(z)\le d_x(y)\}$ is convex.
For $d_{x}(y)$ big enough, these level hypersurfaces $L_y$ do not intersect $O$. 
Therefore, we may assume that $\Omega_{x}\cap O=\emptyset$.
We define
\begin{equation}\label{defL}
\Lambda=\bigcap_{x\in\partial_{\infty}M\setminus\Gamma}\left(  M\setminus\Omega_{x}\right).
\end{equation}
Then $O\subset\Lambda$, $\Lambda$ is closed in $M$ and $\Int\Lambda$ is an open convex subset of $M$ such that 
$\partial_{\infty}\Lambda=\Gamma$. 
\end{proof}
\begin{remark}\label{rmkbar}
It is well known that $\Lambda$ is a barrier for compact area minimizing integral currents with 
boundary in $\Lambda$, that is, if $L$ is a
minimal locally integral current such that $\spt\partial L\subset\Lambda$, then $\spt L\subset\Lambda$; see \cite{AndInv}.
\end{remark}

The next lemma is a variant of \cite[Lemma 2.4]{BL}. It leads to a crucial Harnack-type upper density bound (Proposition~\ref{C}) for minimizing currents 
mod 2. Clearly such a result can not hold for integer multiplicity currents.

\begin{lemma}[Sublinear isoperimetric inequality]
\label{isope} Let $R\in\cR_{k-1}^2(M)$ be a boundary in $\bar{B}(x,r)$, for
some $x\in M$, $r>0$, and $1\leq k\le n$. Then there exist $S\in\cR_{k}^2(M)$ satisfying 
$\spt S\subset \bar{B}(x,r)$ and $\partial S=R$, and
a function $G\colon \R^{+}\to\R^{+}$ depending on
$K_{M},\ n$, and $x$ such that
\[
\mass (S)\leq G(r)\mass(R)^{1-\delta},
\]
where $\delta=1/(n-k+1).$
\end{lemma}

\begin{proof}
We set $\bar{R}=(\exp_{x}^{-1})_{\sharp}R\in\cR_{k-1}^2(T_{x}M),$
where $\exp_{x}\colon T_{x}M\to M$ is the exponential map. Since
$\bar{R}$ bounds in the Euclidean ball $\bar{B}(0,r)$ of $T_{x}M$, the compactness theorem \cite[(4.2.17)$^\nu$]{federer} implies the
existence of a current $\bar{S}\in\cR_{k}^2(T_{x}M)$, with $\partial\bar{S}=\bar{R}$ and $\spt\bar{S}\subset \bar{B}(0,r)$ such that 
$\bar{S}$ is minimizing in $\bar{B}(0,r)$ (with respect to the
Euclidean metric). Since $\bar{B}(0,r)$ is convex, $\bar{S}$ is minimizing in the whole $T_x M$.
The isoperimetric inequality \cite[2.5]{morganH} due to Morgan gives
\[
\mass(\bar{S})\leq c_{0}r^{n\delta}\mass(\bar{R})^{1-\delta},
\]
where $c_{0}$ depends on $k$ and $n$, and $\delta$ is given as in the statement of
the lemma and $\mass$ here stands for the Euclidean mass. Then, taking
$S=(\exp_{x})_{\sharp}\bar{S}\in\cR_{k}^2(M)$ we have
$\spt S\subset \bar{B}(x,r)$ and $\partial S=R$. It is well known that
if $\pi:M\rightarrow P$ is a Lipschitz map (with Lipschitz constant
$\Lip (\pi)$) then we have, for all $T\in\cR_{k}^2(M)$,
\begin{equation}\label{morganiso}
\mass(\pi_{\sharp}T)\leq\Lip (\pi)^{k}\mass(T).
\end{equation}
We then observe that the restriction of $\exp_{x}$ to $\bar{B}(0,r)$ is bi-Lipschitz. More precisely, by the Hessian comparison theorem,
$\exp_{x}|\bar{B}(0,r)$ is Lipschitz with a constant $G(r)$ for some function $G$ depending on 
\[
\max_{\bar{B}(x,r)}|K_{M}|,
\]
and, on the other hand, $\exp_{x}^{-1}$ is Lipschitz with constant $1$ since $K_{M}\leq 0$.
The lemma then follows from \eqref{morganiso}.
\end{proof}
Once we have Lemma~\ref{isope} in use, the following important upper density bound can be proven as in \cite[p. 131-132]{BL}
with only minor changes.
\begin{proposition}
\label{C} Given $x\in M$ and $r>0$ there exists a function $H\colon\R^{+}\to\R^{+}$ 
depending on $K_{M},n$, and $x$ such that, for any $S\in\cR_{k}^2(M)$ that is minimizing in $\bar{B}(x,r)\subset M\setminus\spt\partial S$, 
we have
\begin{equation}\label{upb}
\mass \left(S\llcorner \bar{B}\left(  x,\tfrac{r}{2}\right)  \right)  \leq
H(r).
\end{equation}
\end{proposition}
Before proving Theorem~\ref{TSMC}, we record the following lower density bound that holds also for integer multiplicity currents; see e.g. 
\cite[3.1]{Lang}.
\begin{lemma}[Lower density bound]\label{ldb}
For all $k\in\{2,\ldots,n\}$ there exists a constant $\delta=\delta(k)>0$ such that
\[
\mass\left(S\llcorner \bar{B}(x,r)\right)\ge \delta r^k
\]
whenever $S\in\cR_{k,\loc}^2(M)\cup \cR_{k,\loc}(M)$ is minimizing in $M$, $x\in\spt S$, and $0<r\le \dist(x,\spt\partial S)$.
\end{lemma}
\subsection{Proof of Theorem \ref{TSMC}}
Let $\Gamma\subset \partial_{\infty} M$ be a (topologically) embedded closed $(k-1)$-dimensional submanifold, with $2\leq k\leq n-1$.
Recall that we are looking for a complete, absolutely area minimizing, locally 
rectifiable $k$-current $\Sigma$ modulo 2 in $M$ asymptotic to $\Gamma$ at infinity, i.e. $\partial_{\infty}\spt\Sigma=\Gamma$. 

To prove the existence of such $\Sigma$ we adapt the methods of Bangert and Lang \cite[Section 4]{BL} (see also \cite[Section 5]{Lang}).

\begin{proof}[Proof of Theorem \ref{TSMC}]
As in \cite[p. 139]{BL} (or \cite[p. 44]{Lang}) we approximate $\Gamma$ by a sequence of closed singular Lipschitz chains $\sigma_i$ 
(with $\Z_2$ coefficients) 
such that $\spt\sigma_i\subset \partial B(o,i)\cap \Int\Lambda$ for all $i$ and for every open set 
$V\subset\overline{M}$ meeting $\Gamma$ there exists a closed set $K\subset\overline{M}$, contained in $V\setminus\Gamma$, such that almost all 
$\sigma_i$ do not bound in $M\setminus K$. Each $\sigma_i$ determines a cycle $S_i\in\cR_{k-1}^2(M)$ and all the properties of $\sigma_i$ above hold for 
$S_i$, as well. By the lower semicontinuity of the mass functional and the compactness theorem \cite[(4.2.17)$^\nu$, p. 432]{federer} we find, 
for each $i$, a $k$-current $\Sigma_i\in \cR_{k}^2(M)$ that is minimizing in $\bar{B}(o,i)$ and $\partial \Sigma_{i}=S_i$. Furthermore, 
$\spt\Sigma_i\subset \bar{B}(o,i)\cap\Lambda$ since $\bar{B}(o,i)\cap\Lambda$ is convex. Applying Proposition~\ref{C} we see that for every $r>0$ there 
exists a constant $C(r)<\infty$ such that 
\[
\mass\left(\Sigma_{i}\llcorner \bar{B}(o,r)\right) + \mass\left((\partial\Sigma_{i})\llcorner \bar{B}(o,r)\right)
= \mass\left(\Sigma_{i}\llcorner \bar{B}(o,r)\right) \le C(r)
\] 
for all sufficiently large $i$. Using again the compactness theorem and a diagonal argument we find a subsequence 
$\Sigma_{i_j}$ converging in the local flat topology modulo 2 to a complete, absolutely area minimizing current $\Sigma\in\cR_{k,\loc}^2(M)$. 
(Non-triviality of $\Sigma$ will be proven below.)
Clearly $\spt\Sigma\subset\Lambda$, and therefore 
$\partial_{\infty}\spt\Sigma\subset\partial_{\infty}\Lambda=\Gamma$. It remains to show that $\Gamma\subset\partial_{\infty}\spt\Sigma$.
Let $V\subset\overline{M}$ be an open set meeting $\Gamma$. Then there exists a closed set $K\subset\overline{M}$, contained in $V\setminus\Gamma$, such 
that almost all $S_i$ do not bound in $M\setminus K$. This, together with the fact that $\spt\Sigma_i\subset\Lambda$ for all $i$, imply that almost 
all $\spt\Sigma_{i_j}$ intersect with the compact subset $K\cap\Lambda$ of $M$. Since $\Sigma_{i_j}$ is minimizing in $\bar{B}(o,i_j)$, we 
obtain from Lemma~\ref{ldb} that
\[
\mathbf{M}\left(\Sigma_{i_j}\llcorner \bar{B}(p_{i_j},r)\right)\ge\delta r^k
\]
for all $p_{i_j}\in \spt\Sigma_{i_j}\cap K\cap\Lambda$ and $0<r\le\dist(p_{i_j},\spt\partial\Sigma_{i_j})$. The compactness of $K\cap\Lambda$ implies the 
existence of $p_0\in K\cap\Lambda,\ r_0>0$, and a subsequence of $\Sigma_{i_j}$, still denoted by $\Sigma_{i_j}$, such that
\[
\mathbf{M}\left(\Sigma_{i_j}\llcorner \bar{B}(p_0,r)\right)\ge \delta r^k
\] 
for all $0<r\le r_0$.
We conclude that 
\[
\mathbf{M}\left(\Sigma\llcorner \bar{B}(p_0,r+\epsilon)\right)\ge \delta r^k
\]
for all $\epsilon>0$ and $0<r\le r_0$, and therefore $p_0\in\spt\Sigma\cap K\cap\Lambda\subset V$ (see \cite[5.4.2]{federer}). 
Since this holds for all open sets $V\subset\overline{M}$ meeting $\Gamma$, we finally obtain $\Gamma\subset\partial_{\infty}\spt\Sigma$.

The complete, absolutely area minimizing current $\Sigma\in\cR_k^2(M)$, with $\partial_\infty\spt\Sigma=\Gamma$, then solves the asymptotic 
Plateau problem. Hence Theorem~\ref{TSMC} holds true.
\end{proof}

In fact, the above (with minors changes) proves also the following analogue to the general existence result \cite[4.2]{BL}.
\begin{theorem}\label{TSMCII}
Suppose that $M$ is an $n$-dimensional Cartan-Hadamard manifold satisfying the SC condition.
Let $1\le k<n$ and let $\Gamma\subset\partial_{\infty}M$ be closed such that there exists a sequence of boundaries $S_i\in\cR_{k-1}^2(M)$ in $M$
satisfying the following conditions:
\begin{enumerate}
\item[(i)] For every neighborhood $U$ of $\Gamma$ in $\overline{M}$ there exists $k>0$ such that $\spt S_i\subset U$ for all $i\ge k$, and
\item[ii)] for every open $V\subset\overline{M}$ meeting $\Gamma$ there exists a closed set $K\subset V\setminus\Gamma$ such that almost all $S_i$
do not bound in $M\setminus K$.
\end{enumerate}
Then there exists a complete $k$-dimensional surface $\Sigma\in\cR_{k,\loc}^2(M)$ which is minimizing in $M$ and $\partial_{\infty}\spt\Sigma=\Gamma$.
\end{theorem}

\subsection{Proof of Theorem \ref{TSMC2}}
Since the proof is very similar to that of \cite[3.2]{Lang1}, we just present a brief sketch. 
\begin{proof}[Proof of \ref{TSMC2}]
We assume that $(M,g)$ is a Cartan-Hadamard manifold of dimension $n\ge 2$ satisfying the SC condition.
Let $\Gamma$ be a non-empty subset of $\partial_{\infty}M$ as in Theorem~\ref{TSMC2}, that is $\Gamma=\bd A$ for some 
$A\subset\partial_{\infty}M$ with $A=\cl(\ir A)$.
We identify $\partial_{\infty}M$ with the unit sphere $S_o M\subset T_o M$ for a fixed $o\in M$. 
Then we notice that the Riemannian metric $g$ when restricted to closed geodesic balls $\bar{B}(o,r)$ is bi-Lipschitz equivalent to
the standard hyperbolic metric with some constants $\beta(r)\ge \alpha(r)>0$. 
For each $i\in\N$, let 
\[
T_i=[S(o,i)\cap\cone_o (\ir A)]\in\cR_{n-1}(M)
\]
equipped with the canonical orientation of $S(o,i)$. Then $\spt\partial T_i\subset\cone_o\Gamma$. As in \cite[p. 32]{Lang1} we deduce that, 
for each $i$, there exists a set
$W_i\subset \bar{B}(o,i)$ of finite perimeter such that the rectifiable $(n-1)$-current $S_i:=\partial[W_i]-T_i$ is minimizing in 
$\bar{B}(o,i)$. Thus we can bound the total variation (measure) from above by
\begin{equation}
\sup_{i}\norm{\partial[W_i]}\bigl(B(o,r)\bigr)\le F(r)
\end{equation}
for every $r>0$, where $F\colon\R^{+}\to\R^{+}$ depends on $n$ and the functions $\alpha$ and $\beta$ above. Hence there exist a subsequence
$\{W_{i_j}\}$ of  $\{W_{i}\}$ and a set $W$ of locally finite perimeter such that $[W_{i_j}]\to [W]$ weakly. It follows that 
$\Sigma=\partial[W]$ is a locally rectifiable $(n-1)$-current that is minimizing in $M$. It remains to show that 
$\partial_{\infty}\spt\Sigma=\Gamma$ and $\partial_{\infty}W=A$. Recall from \eqref{defL} the definition of the convex set $\Lambda$ corresponding to
$\Gamma$. First we observe that $\spt S_i\subset\Lambda$ since 
$\spt \partial S_i=\spt\partial T_i\subset\cone_o\Gamma\subset\Lambda$ and $S_i$ is minimizing in $\bar{B}(o,i)$ hence also in the 
convex set $\bar{B}(o,i)\cap\Lambda$. It follows that $\spt\Sigma\subset\Lambda$ and hence $\partial_{\infty}\spt\Sigma\subset\Gamma$.
The remaining claims $\Gamma\subset\partial_{\infty}\spt\Sigma$ and $\partial_{\infty}W=A$ can be proven exactly as in \cite[p. 55]{Lang1}.
\end{proof}

\subsection{Proof of Theorem \ref{TSMC3}}
Let us recall the assumptions in Theorem~\ref{TSMC3}.
Let $M^{n},\ n\geq 3$, be a rotationally symmetric Cartan-Hadamard manifold around $o\in M$ satisfying the SC condition. 
Identify $\partial_{\infty}M^n$ with the unit $(n-1)$-sphere $S_o M^n\subset T_o M^n$.
Suppose that $\Gamma\subset S_o M^n$ is a closed, smoothly embedded, orientable $(k-1)$-dimensional submanifold, with $2\leq k\leq n-1$. 

We want to prove the existence of a complete, absolutely area minimizing, locally rectifiable $k$-current $\Sigma$ in $M^n$ that is asymptotic 
to $\Gamma$ at infinity, i.e. $\partial_{\infty}\spt\Sigma=\Gamma$. 

Let $f\colon [0,\infty)\to [0,\infty)$ be the smooth function such that the Riemannian metric of $M$
is given in polar coordinates $(r,\vartheta)$, with the pole at $o$, by
\begin{equation}\label{rm}
ds^2=dr^2 + f^2(r)d\vartheta^2.
\end{equation}
If $\bar{B}^k(o,t)=\exp_o \bar{B}^k(0,t)\subset M$ is a geodesic $k$-dimensional (closed) ball, 
then its $k$-dimensional volume 
is given by  
\[
\Vol_k(\bar{B}^k(o,t))= k\alpha_k\int_0^t f^{k-1}(s)\,ds,
\]
where $\alpha_k$ is the volume of the $k$-dimensional Euclidean unit ball $B^k$.

We need the following monotonicity formula for the mass ratio in order to obtain a 
counterpart to the crucial upper density bound \eqref{upb}. It is essential that the center of the ball is the 
pole $o\in M$.
\begin{lemma}
\label{mf} Suppose that $M^{n},\ n\geq 3$, is a rotationally symmetric Cartan-Hadamard manifold around $o\in M$.
Let $S\in\cR_{k,\loc}(M),\ 2\le k\le n-1$, be absolutely area minimizing in $B(o,r)\subset M\setminus\spt\partial S$.
Then the function
\begin{equation}\label{mni}
t\mapsto\dfrac{\mass(S\llcorner \bar{B}(o,t))}{\Vol_k(\bar{B}^k(o,t))} 
\end{equation}
is non-decreasing on the interval $(0,r]$.
\end{lemma}
\begin{proof}
The proof is similar to that of \cite[9.3]{morgan} and \cite[2.2]{BL}. For $0<t \le r$, we write
$m(t)=\mass(S\llcorner\bar{B}(o,t))$. Since $m$ is increasing, $m'(t)$ exists for a.e. $t\in (0,r]$. Slicing by the distance function
$\rho(x)=d(x,o)$ yields
\begin{equation}\label{slice}
\mass\bigl(\partial(S\llcorner\bar{B}(o,t))\bigr)\le m'(t)
\end{equation}
for a.e. $t\in (0,r]$; see \cite[28.9]{simon} and \cite[1.4]{BL}. Hence $R_t:=\partial(S\llcorner\bar{B}(o,t))$ is a rectifiable 
$(k-1)$-current for all such $t$. Thus there exist a $(k-1)$-rectifiable set $V_t\subset S(o,t)$ and a Borel-measurable simple $(k-1)$-vector field 
$\eta$, with integer (multiplicity) $\abs{\eta}$ and $\spt\eta\subset S(o,t)$, such that $R_t=[V_t]$ (with multiplicity $\abs {\eta}$). That is,
\[
R_t(\omega)=\int_{V_t}\omega(\eta)\,d\cH^{k-1}
\]
for differential $(k-1)$-forms $\omega$; see \cite[4.4]{morgan}. Next we extend $\eta$ radially along geodesic rays from $o$ to the simple 
$(k-1)$-vector field $\tilde\eta$ in $C_t:=\cone_o V_t$. This is possible since $\spt\eta\subset S(o,t)$. 
Then $\nabla\rho\wedge\tilde{\eta}$ is a Borel-measurable simple $k$-vector field on $C_t$, with integer multiplicity $\abs{\tilde{\eta}}$. 
Hence $S_t:=[C_t]$ (with multiplicity $\abs{\tilde{\eta}}$), 
\[
S_t(\omega)=\int_{C_t}\omega(\tilde{\eta})\,d\cH^{k},
\]
is a rectifiable $k$-current, with $\partial S_t=R_t$.
Concluding as in \cite[p. 129]{BL} and taking into account that $S$ is minimizing and the Riemannian metric is given by \eqref{rm}, we obtain
using \eqref{slice} that
\begin{equation}\label{diffineq}
m(t)=\mass(S\llcorner\bar{B}(o,t))\le \mass(S_t)=\frac{\beta_k(t)}{\beta_k^\prime(t)}\mass(R_t)\le\frac{\beta_k(t)}{\beta_k^\prime(t)}m'(t)
\end{equation}
for a.e. $t\in (0,r]$, where $\beta_k(t)=\Vol_k(\bar{B}^k(o,t))$. Hence
\[
\frac{d}{dt}\left(\frac{m(t)}{\beta_k(t)}\right)\ge 0
\]
a.e. on $(0,r]$ and the claim follows.
\end{proof}
\begin{proof}[Proof of Theorem~\ref{TSMC3}]
For each $t>0$ we define 
\[
K_t=\cone_o\Gamma\cap\bar{B}(o,t)\text{ and }
\Gamma_t=\cone_o\Gamma\cap S(o,t).
\]
Since $\Gamma\hookrightarrow S_o M$ is a smooth embedding, each $K_t$ and $\Gamma_t$ define rectifiable currents
$[K_t]\in\cR_{k}(M)$ and $[\Gamma_t]\in\cR_{k-1}(M)$, respectively. Furthermore, since $\Gamma$ is orientable, we may assume
that $[\Gamma_t]=\partial[K_t]$. 
Then the $k$-dimensional volume of $K_t$ is given by 
\[
\Vol_k(K_t) = c\int_0^t f^{k-1}(s)\,ds <\infty,
\]
where $c=c(\Gamma)$ is a constant. 
We obtain
\[
\Vol_k(K_t)=c_k\Vol_k\bigl(\bar{B}^k(o,t)\bigr),
\]
where the constant $c_k$ depends only on $k$ and $\Gamma$. For each $i\in\N$, let $\Sigma_i\in\cR_k(M)$ be an integral current in $M$, with 
$\partial\Sigma_i=[\Gamma_i]=\partial[K_i]$, that is minimizing in $\bar{B}(o,i)$. Let $\Lambda$ be given by \eqref{defL} corresponding to
$\Gamma$. Note that $\cone_o\Gamma\subset\Lambda$. Since $\bar{B}(o,i)\cap\Lambda$ is convex, $\Sigma_i$ is minimizing in $\bar{B}(o,i)\cap\Lambda$.
Hence
\begin{equation}\label{coneball}
\mass(\Sigma_i)\le \Vol_k(K_i)=c_k\Vol_k\bigl(\bar{B}^k(o,i)\bigr).
\end{equation}
Using the monotonicity formula (Lemma~\ref{mf}) we obtain 
\begin{equation}\label{upb2}
\mass(\Sigma_i\llcorner \bar{B}(o,r))\le c_k\Vol_k\bigl(\bar{B}^k(o,r)\bigr)
\end{equation}
for all $r>0$ and $i\ge r$. This serves as a counterpart to the upper density bound \eqref{upb}. 
Hence using the compactness theorem and a diagonal argument we can extract a subsequence $\Sigma_{i_j}$ converging weakly to 
a (non-trivial) complete, absolutely area minimizing current $\Sigma\in\cR_{k,\loc}(M)$, with $\partial_{\infty}\spt\Sigma\subset\Gamma$. 
The inclusion $\Gamma\subset\partial_{\infty}\spt\Sigma$ holds true since the sequence of boundaries $[\Gamma_i]\in\cR_{k-1}(M)$ satisfies
conditions (i) and (ii) in Theorem~\ref{TSMCII}; cf. \cite[p. 139]{BL}.
\end{proof}
 \subsection{Proof of Theorem \ref{TSMC4}}
The proof is similar to that of Theorem~\ref{TSMC3} with some changes. The following lemma is crucial in order to obtain a useful counterpart of 
upper density bounds \eqref{upb}, \eqref{upb2},  cf. \cite{RTo}.
 \begin{lemma}\label{implemma}
Let $a,b\colon [0,\infty)\to [0,\infty)$ be smooth, increasing functions such that $b\ge a$ and
\begin{equation}\label{intcond}
I:=\int_{0}^{\infty}\dfrac{b^2(t)-a^2(t)}{b(t)}\,dt<\infty.
\end{equation}
If $f_a$ and $f_b$ are the solutions of the initial value problem \eqref{Jacobi_eq} with $k=a$ and $k=b$, respectively, then there exists a constant
$c=c(I)$ such that $f_b(t)\le c\,f_a(t)$ for all $t\ge 0$.
\end{lemma}
\begin{proof}
For all $0<\delta<r$, we have
\begin{align*}
\left\vert \log\frac{f_b(r)}{f_{a}(r)}-\log\frac{f_b(\delta)}{f_{a}(\delta)}\right\vert  
&  =\left\vert \int_{\delta}^{r}\left(  \frac{f_b^{\prime}(s)f_a(s)-f_{a}^{\prime}(s)f_b(s)}{f_b(s) f_{a}(s)}\right)\,ds\right\vert \\
&  \leq\int_{\delta}^{r}\int_{0}^{s}\frac{\left(b^2(t)-a^2(t)\right)}{f_b(s)f_a(s)  } f_b(t) f_a (t)\,dt\,ds.
\end{align*}
where the last estimate follows from the identity
\[
\left(f_b^\prime f_a-f_a^\prime f_b\right)^\prime =f_b^{\prime\prime}f_a-f_a^{\prime\prime}f_b=(b^2-a^2)f_af_b.
\]
For fixed $t$, let $f$ be the solution of the initial value problem
\[
f^{\prime\prime}(s)-a^2(t)f(s)=0,\text{ }f(t)=f_{a}(t),\text{ }f^{\prime}(t)=f_{a}^{\prime}(t).
\]
Since $a$ is increasing, we conclude that for $s\geq t$%
\begin{align*}
f_{a}(s)  &  \geq f(s)=f_{a}(t)\cosh\bigl(a(t)(s-t)\bigr) 
  +\frac{f_{a}^{\prime}(t)}{a(t)}\sinh\bigl(a(t)(s-t)\bigr) \\
&  \geq f_{a}(t)\cosh\bigl(a(t)(s-t)  \bigr)  .
\end{align*}
Similarly, 
\[
f_b(s)\geq f_{b}(t)\cosh\bigl(b(t)(s-t)\bigr)
\]
for $s\ge t$. By Fubini's theorem, we obtain for all $0<\delta<r$
\begin{align*}
&\left\vert \log\frac{f_b(r)}{f_{a}(r)}-\log\frac{f_b(\delta)}{f_{a}(\delta)}\right\vert  \\
&  \leq\int_{0}^{\infty}
\left(b^2(t)-a^2(t)\right)f_b(t) f_a (t)\left(\int_{t}^{\infty}\frac{ds}{f_b(s)f_a(s)}\right)dt\\
&  \leq\int_{0}^{\infty}\left(b^2(t)-a^2(t)\right)
\left( \int_{t}^{\infty}\frac{ds}{\cosh\bigl(a(t)(s-t)\bigr)\cosh\bigl(b(t)(s-t)\bigr)}\right)dt\\
&  \leq \int_{0}^{\infty}\left(b^2(t)-a^2(t)\right) 
\left[  \frac{\arctan\bigl(\sinh b(t)(s-t)\bigr)}{b(t)}\right]_{t}^{\infty}dt\\
&  =\frac{\pi}{2}\int_{0}^{\infty}\frac{\left( b^2(t)-a^2(t)\right)}{b(t)}dt=:C<+\infty.
\end{align*}
Since 
\[
\lim_{\delta\to 0}\frac{f_b(\delta)}{f_{a}(\delta)}=1,
\] 
it follows that $f_b(t)\leq c\, f_a(t)$, for all $t\geq 0$. 
\end{proof}
\begin{lemma}
\label{mf2} Suppose that $M^{n},\ n\geq 3$, is a Cartan-Hadamard manifold such that sectional curvatures have an upper bound
\[
K(P)\le -a^2(\rho(x)),
\]
with a smooth function $a\colon [0,\infty)\to [0,\infty)$.
Let $S\in\cR_{k,\loc}(M),\ 2\le k\le n-1$, be absolutely area minimizing in $B(o,r)\subset M\setminus\spt\partial S$.
Then the function
\begin{equation}\label{mni2}
t\mapsto\dfrac{\mass(S\llcorner \bar{B}(o,t))}{\alpha_k\int_{0}^tf_a^{k-1}(s)ds} 
\end{equation}
is non-decreasing on the interval $(0,r]$.
\end{lemma}
\begin{proof}
The proof is similar to that of Lemma~\ref{mni} with the cone inequality \eqref{diffineq} replaced by
\[
m(t)=\mass(S\llcorner\bar{B}(o,t))\le \mass(S_t)\le\frac{\beta_k(t)}{\beta_k^\prime(t)}\mass(R_t)\le\frac{\beta_k(t)}{\beta_k^\prime(t)}m'(t)
\]
for a.e. $t\in (0,r]$, where 
\[
\beta_k(t)=\alpha_k\int_{0}^tf_a^{k-1}(s)ds.
\] 
\end{proof}
\begin{proof}[Proof of Theorem \ref{TSMC4}]
We use the same notation as in the proof of Theorem~\ref{TSMC3}.
For each $t>0$ we define 
\[
K_t=\cone_o\Gamma\cap\bar{B}(o,t)\text{ and }
\Gamma_t=\cone_o\Gamma\cap S(o,t).
\]
%
Now the $k$-dimensional volume of $K_t$ has an upper bound 
\[
\Vol_k(K_t) \le c\int_0^t f_b^{k-1}(s)\,ds <\infty,
\]
where $c=c(\Gamma)$ is a constant. Applying Lemma~\ref{implemma} we get 
\[
\Vol_k(K_t)\le c_k \int_0^t f_a^{k-1}(s)\,ds,
\]
where the constant $c_k$ depends only on $k,\ \Gamma$, and the functions $a$ an $b$. 
For each $i\in\N$, let $\Sigma_i\in\cR_k(M)$ be an integral current in $M$, with 
$\partial\Sigma_i=[\Gamma_i]=\partial[K_i]$, that is minimizing in $\bar{B}(o,i)$. 
Then
\[
\mass(\Sigma_i)\le \Vol_k(K_i)\le c_k \int_0^i f_a^{k-1}(s)\,ds.
\]
Using the monotonicity formula (Lemma~\ref{mf2}) we obtain an upper density bound
\[
\mass(\Sigma_i\llcorner \bar{B}(o,r))\le c_k\int_0^r f_a^{k-1}(s)\,ds
\]
for all $r>0$ and $i\ge r$. The rest of the proof is similar to that of Theorems~\ref{TSMC2} and \ref{TSMC3}. 
\end{proof}


\begin{thebibliography}{10}

\bibitem{alm}
{\sc Almgren, Jr., F.~J.}
\newblock {$Q$} valued functions minimizing {D}irichlet's integral and the
  regularity of area minimizing rectifiable currents up to codimension two.
\newblock {\em Bull. Amer. Math. Soc. (N.S.) 8}, 2 (1983), 327--328.

\bibitem{ancrevista}
{\sc Ancona, A.}
\newblock Convexity at infinity and {B}rownian motion on manifolds with
  unbounded negative curvature.
\newblock {\em Rev. Mat. Iberoamericana 10}, 1 (1994), 189--220.

\bibitem{AndInv}
{\sc Anderson, M.~T.}
\newblock Complete minimal varieties in hyperbolic space.
\newblock {\em Invent. Math. 69}, 3 (1982), 477--494.

\bibitem{AndCMH}
{\sc Anderson, M.~T.}
\newblock Complete minimal hypersurfaces in hyperbolic {$n$}-manifolds.
\newblock {\em Comment. Math. Helv. 58}, 2 (1983), 264--290.

\bibitem{andJDG}
{\sc Anderson, M.~T.}
\newblock The {D}irichlet problem at infinity for manifolds of negative
  curvature.
\newblock {\em J. Differential Geom. 18}, 4 (1983), 701--721 (1984).

\bibitem{ATU}
{\sc Arnaudon, M., Thalmaier, A., and Ulsamer, S.}
\newblock Existence of non-trivial harmonic functions on {C}artan-{H}adamard
  manifolds of unbounded curvature.
\newblock {\em Math. Z. 263}, 2 (2009), 369--409.

\bibitem{BL}
{\sc Bangert, V., and Lang, U.}
\newblock Trapping quasiminimizing submanifolds in spaces of negative
  curvature.
\newblock {\em Comment. Math. Helv. 71}, 1 (1996), 122--143.

\bibitem{borbpams}
{\sc Borb{\'e}ly, A.}
\newblock A note on the {D}irichlet problem at infinity for manifolds of
  negative curvature.
\newblock {\em Proc. Amer. Math. Soc. 114}, 3 (1992), 865--872.

\bibitem{Bor}
{\sc Borb{\'e}ly, A.}
\newblock The nonsolvability of the {D}irichlet problem on negatively curved
  manifolds.
\newblock {\em Differential Geom. Appl. 8}, 3 (1998), 217--237.

\bibitem{CHRv}
{\sc Casteras, J.-B., Holopainen, I., and Ripoll, J.~B.}
\newblock In progress.

\bibitem{CHR1}
{\sc Casteras, J.-B., Holopainen, I., and Ripoll, J.~B.}
\newblock Asymptotic {D}irichlet problem for {${\mathcal A}$}-harmonic and
  minimal graph equations in a {C}artan-{H}adamard manifolds.
\newblock {\em Comm. Anal. Geom.\/} (To appear).

\bibitem{choi}
{\sc Choi, H.~I.}
\newblock Asymptotic {D}irichlet problems for harmonic functions on
  {R}iemannian manifolds.
\newblock {\em Trans. Amer. Math. Soc. 281}, 2 (1984), 691--716.

\bibitem{cosku}
{\sc Coskunuzer, B.}
\newblock Asymptotic {P}lateau problem: a survey.
\newblock In {\em Proceedings of the {G}\"okova {G}eometry-{T}opology
  {C}onference 2013\/} (2014), G\"okova Geometry/Topology Conference (GGT),
  G\"okova, pp.~120--146.

\bibitem{cosku2}
{\sc Coskunuzer, B.}
\newblock Asymptotic {P}lateau problem in {$\Bbb H^2\times \Bbb R$}.
\newblock {\em Preprint arXiv:1604.01498 [math.DG]\/} (2016).

\bibitem{DeLe}
{\sc De~Lellis, C.}
\newblock The size of the singular set of area-minimizing currents.
\newblock In {\em Surveys in differential geometry 2016. {A}dvances in geometry
  and mathematical physics}, vol.~21 of {\em Surv. Differ. Geom.} Int. Press,
  Somerville, MA, 2016, pp.~1--83.

\bibitem{DeLeSp}
{\sc De~Lellis, C., and Spadaro, E.}
\newblock Regularity of area minimizing currents {I}: gradient {$L^p$}
  estimates.
\newblock {\em Geom. Funct. Anal. 24}, 6 (2014), 1831--1884.

\bibitem{EO}
{\sc Eberlein, P., and O'Neill, B.}
\newblock Visibility manifolds.
\newblock {\em Pacific J. Math. 46\/} (1973), 45--109.

\bibitem{federer}
{\sc Federer, H.}
\newblock {\em Geometric measure theory}.
\newblock Die Grundlehren der mathematischen Wissenschaften, Band 153.
  Springer-Verlag New York Inc., New York, 1969.

\bibitem{federer2}
{\sc Federer, H.}
\newblock The singular sets of area minimizing rectifiable currents with
  codimension one and of area minimizing flat chains modulo two with arbitrary
  codimension.
\newblock {\em Bull. Amer. Math. Soc. 76\/} (1970), 767--771.

\bibitem{H_ns}
{\sc Holopainen, I.}
\newblock Nonsolvability of the asymptotic {D}irichlet problem for the
  {$p$}-{L}aplacian on {C}artan-{H}adamard manifolds.
\newblock {\em Adv. Calc. Var. 9}, 2 (2016), 163--185.

\bibitem{HR_ns}
{\sc Holopainen, I., and Ripoll, J.~B.}
\newblock Nonsolvability of the asymptotic {D}irichlet problem for some
  quasilinear elliptic {PDE}s on {H}adamard manifolds.
\newblock {\em Rev. Mat. Iberoam. 31}, 3 (2015), 1107--1129.

\bibitem{HoVa}
{\sc Holopainen, I., and V{\"a}h{\"a}kangas, A.}
\newblock Asymptotic {D}irichlet problem on negatively curved spaces.
\newblock {\em J. Anal. 15\/} (2007), 63--110.

\bibitem{kasue}
{\sc Kasue, A.}
\newblock A {L}aplacian comparison theorem and function theoretic properties of
  a complete {R}iemannian manifold.
\newblock {\em Japan. J. Math. (N.S.) 8}, 2 (1982), 309--341.

\bibitem{KloMaz}
{\sc Kloeckner, B. t.~R., and Mazzeo, R.}
\newblock On the asymptotic behavior of minimal surfaces in {$\Bbb H^2\times
  \Bbb R$}.
\newblock {\em Indiana Univ. Math. J. 66}, 2 (2017), 631--658.

\bibitem{Lang1}
{\sc Lang, U.}
\newblock The existence of complete minimizing hypersurfaces in hyperbolic
  manifolds.
\newblock {\em Internat. J. Math. 6}, 1 (1995), 45--58.

\bibitem{Lang}
{\sc Lang, U.}
\newblock The asymptotic {P}lateau problem in {G}romov hyperbolic manifolds.
\newblock {\em Calc. Var. Partial Differential Equations 16}, 1 (2003), 31--46.

\bibitem{morganH}
{\sc Morgan, F.}
\newblock Harnack-type mass bounds and {B}ernstein theorems for area-minimizing
  flat chains modulo {$\nu$}.
\newblock {\em Comm. Partial Differential Equations 11}, 12 (1986), 1257--1283.

\bibitem{morgan}
{\sc Morgan, F.}
\newblock {\em Geometric measure theory}, fourth~ed.
\newblock Elsevier/Academic Press, Amsterdam, 2009.
\newblock A beginner's guide.

\bibitem{RT}
{\sc Ripoll, J., and Telichevesky, M.}
\newblock Regularity at infinity of {H}adamard manifolds with respect to some
  elliptic operators and applications to asymptotic {D}irichlet problems.
\newblock {\em Trans. Amer. Math. Soc. 367}, 3 (2015), 1523--1541.

\bibitem{RTo}
{\sc Ripoll, J., and Tomi, F.}
\newblock Complete minimal discs in {H}adamard manifolds.
\newblock {\em Adv. Calc. Var. 10}, 4 (2017), 315--330.

\bibitem{simon}
{\sc Simon, L.}
\newblock {\em Lectures on geometric measure theory}, vol.~3 of {\em
  Proceedings of the Centre for Mathematical Analysis, Australian National
  University}.
\newblock Australian National University, Centre for Mathematical Analysis,
  Canberra, 1983.

\bibitem{Va_lic}
{\sc V{\"a}h{\"a}kangas, A.}
\newblock Bounded {$p$}-harmonic functions on models and {C}artan-{H}adamard
  manifolds.
\newblock Unpublished licentiate thesis, Department of Mathematics and
  Statistics, University of Helsinki, 2006.

\bibitem{warner}
{\sc Warner, F.~W.}
\newblock Extensions of the {R}auch comparison theorem to submanifolds.
\newblock {\em Trans. Amer. Math. Soc. 122\/} (1966), 341--356.

\end{thebibliography}

\end{document}